\pdfoutput=1
\documentclass{siamltex}
\usepackage{amsfonts,amsmath,mathrsfs}
\usepackage{amssymb}
\usepackage{hyperref}
\usepackage{graphicx}
\usepackage{wrapfig}
\usepackage{subfig}
\usepackage{epsfig}
\usepackage{float}
\usepackage{bbm}

\newcommand{\N}{\mathbb{N}}
\newcommand{\R}{\mathbb{R}}

\newtheorem{remark}[theorem]{Remark}

\newcommand{\dd}{\textrm{d}}

\title{Polynomial collocation for handling an inaccurately known measurement configuration in electrical impedance tomography}

\author{N. Hyv\"onen\footnotemark[2]
\and V. Kaarnioja\footnotemark[2]
\and L. Mustonen\footnotemark[2]
\and S. Staboulis\footnotemark[3]
}

\begin{document}
\maketitle

\renewcommand{\thefootnote}{\fnsymbol{footnote}}

\footnotetext[2]{Aalto University, Department of Mathematics and Systems Analysis, P.O. Box 11100, FI-00076 Aalto, Finland (nuutti.hyvonen@aalto.fi, vesa.kaarnioja@aalto.fi, lauri.mustonen@aalto.fi). The work of these authors was supported by the Academy of Finland (decision 267789). The work of LM was also supported by the Finnish Foundation for Technology Promotion TES.}

\footnotetext[3]{Technical University of Denmark, Department of Applied Mathematics and Computer Science, Asmussens Alle, Building 322, DK-2800, Kgs. Lyngby, Denmark (ssta@dtu.dk). The work of this author was supported by the Danish Council for Independent Research (grant 4002-00123).}

\begin{abstract}
The objective of electrical impedance tomography is to reconstruct the internal conductivity of a physical body based on measurements of current and potential at a finite number of electrodes attached to its boundary. Although the conductivity is the quantity of main interest in impedance tomography, a real-world measurement configuration includes other unknown parameters as well: the information on the contact resistances, electrode positions and body shape is almost always incomplete. In this work, the dependence of the electrode measurements on all aforementioned model properties is parametrized via polynomial collocation. The availability of such a parametrization enables efficient simultaneous reconstruction of the conductivity and other unknowns by a Newton-type output least squares algorithm, which is demonstrated by two-dimensional numerical experiments based on both noisy simulated data and experimental data from two water tanks.
\end{abstract}

\renewcommand{\thefootnote}{\arabic{footnote}}

\begin{keywords}
Electrical impedance tomography, polynomial collocation, uncertainty quantification, Bayesian inversion, inaccurate measurement model, complete electrode model
\end{keywords}

\begin{AMS}
65N21, 35R30, 35R60, 60H15
\end{AMS}

\pagestyle{myheadings}
\thispagestyle{plain}
\markboth{N. HYV\"ONEN, V. KAARNIOJA, L. MUSTONEN, AND S. STABOULIS}{POLYNOMIAL COLLOCATION IN EIT}

\section{Introduction}
\label{sec:introduction}
The objective of {\em electrical impedance tomography} (EIT) is to reconstruct the conductivity/admittivity inside a physical body from boundary measurements of electric current and electromagnetic potential. EIT can be applied to,~e.g.,~medical imaging, process tomography, and nondestructive testing of materials \cite{Borcea02,Cheney99,Uhlmann09}. The most accurate way to model the measurements of EIT is employing the {\em complete electrode model} (CEM), which takes into account the electrode shapes and contact resistances/impedances caused by resistive layers at electrode-object interfaces~\cite{Cheng89,Somersalo92}. 

When EIT is used in practice, the conductivity is typically not the only unknown. In particular, the electrode positions, the contact resistances and the shape of the imaged object are also subject to uncertainties. For example, in a medical application the body shape and the contact resistances obviously depend on the patient, and one cannot assume precise information on the positioning of the electrodes. As it is well known that even slight mismodelling usually ruins the reconstruction of the conductivity in absolute EIT imaging~\cite{Barber88, Breckon88, Kolehmainen97}, not being able to account for such inaccuracies considerably hampers establishing EIT as a practical imaging modality. Since the contact resistances and the electrode locations can be (stably) estimated at the same time as the conductivity reconstruction is formed by a Newton-type algorithm~\cite{Darde12,Vilhunen02}, the most challenging of the aforementioned three sources of uncertainty is arguably the inaccurately known object shape. In the following, we present a brief survey of the previously introduced methods for recovering from uncertainties in the exterior boundary shape in EIT; for a more comprehensive discussion, see~\cite{Nissinen11b}.

{\em Difference} imaging is the simplest technique for handling uncertainties in the measurement set-up of EIT~\cite{Barber84}: Electrode measurements are performed at two time instants (or angular frequencies~\cite{Alberti16}) and the corresponding change in the conductivity (or admittivity) is reconstructed. The main idea is that the modelling errors partly cancel out when the difference data are formed, assuming there are no alterations in the boundary shape in between the two measurements, e.g., due to the breathing cycle of a patient. On the negative side, difference imaging is highly approximative as the theoretical grounds for its functionality rely on a linearization of the forward model. In addition, difference data are not always available. 

The first generic algorithm capable of coping with an unknown object boundary in {\it absolute} EIT imaging was introduced in two spatial dimensions by Kolehmainen, Lassas and Ola~\cite{Kolehmainen05,Kolehmainen07}. Allowing an oversimplification, their approach is based on compensating for the mismodelled geometry by reconstructing a (slightly) anisotropic conductivity. An obvious weakness of the ideas in~\cite{Kolehmainen05,Kolehmainen07} is the difficulty in generalizing the corresponding \emph{numerical algorithm} to three dimensions. The so-called approximation error methodology~\cite{Kaipio05} was successfully applied to EIT with an inaccurately known boundary shape in~\cite{Nissinen11,Nissinen11b}: The error caused by the uncertainties in the model geometry (and other nuisance parameters) is represented as an auxiliary measurement noise process whose second order statistics are approximated via simulations based on the prior probability models for the conductivity and the boundary shape. Subsequently, a reconstruction of the conductivity is formed within the Bayesian paradigm. The most straightforward approach to dealing with an inexactly known body shape in EIT was introduced in~\cite{Darde13a,Darde13b}, where the Fr\'echet derivative of the solution to the CEM with respect to the exterior boundary shape was employed in a regularized Newton-type output least squares algorithm that simultaneously reconstructs the conductivity, the contact resistances, the electrode positions and the exterior boundary of the imaged object. The main weakness of the algorithm in~\cite{Darde13a,Darde13b} lies with the numerical instability in the computation of the needed shape derivatives, which necessitates the use of relatively dense {\em finite element} (FE) meshes and thus slows down the computations to a certain extent.

This work tackles absolute EIT imaging with an unknown object shape by means of (stochastic) polynomial collocation. The conductivity, the contact resistances, the electrode positions and the boundary shape are parametrized by a finite number of, say a thousand, parameters supported in a bounded interval; in the framework of stochastic collocation~\cite{Babuska10}, these parameters would be interpreted as uniformly distributed random variables. The forward problem of the CEM is then treated as a parametric elliptic boundary value problem whose solution depends not only on the current feed and the spatial variable but also on the high-dimensional parameter vector. This forward problem is solved by a (stochastic) {\em collocation finite element method} (cFEM)~\cite{Babuska10}: The standard CEM problem is first solved with a {\em finite element method} (FEM) for the conductivities and measurement settings defined by an appropriate sparse grid of collocation points in the parameter hypercube, and subsequently the dependence of the forward solution on the parameters is generalized to the whole hypercube via collocation by tensor products of Legendre polynomials. In particular, such a procedure gives an approximate parametrization of the electrode potentials with respect to (the parameters defining) the conductivity, the contact resistances, the electrode positions and the object shape, which makes it possible to reconstruct these unknowns,~e.g.,~by Tikhonov regularization or via Bayesian inference. Indeed, the described approach results in a functional reconstruction algorithm that is tested both with simulated and experimental data in a two-dimensional setting. See \cite{Hiptmair15} for a closely related algorithm as well as related theory in inverse obstacle scattering. 

Compared to the previous methods for recovering from uncertainties in the geometric specifications of the measurement set-up in EIT, on a general level the introduced algorithm most closely resembles the approximation error technique employed in~\cite{Nissinen11,Nissinen11b}: Both require heavy off-line computations that can be performed prior to the measurements (to simulate the statistics of the approximation error process or to parametrize the dependence on the unknowns via polynomial collocation), but both also allow a fast on-line reconstruction phase once the measurements become available. For completeness, it should be mentioned that \cite{Hakula14,Hyvonen15} used a stochastic {\em Galerkin} FEM (cf.,~e.g.,~\cite{Schwab11a}) as a building block of a Bayesian reconstruction algorithm for EIT under the assumption that the electrode positions and the object shape are known. However, it seems difficult to apply a stochastic Galerkin FEM to handling uncertainties in the measurement configuration of EIT; see,~e.g.,~\cite{Hiptmair15} for similar conclusions.

This text is organized as follows. Section~\ref{sec:CEM} recalls the CEM and presents its parametric extension, while Section~\ref{sec:cFEM} describes how cFEM can be applied to the CEM. The actual implementation of the reconstruction algorithm is discussed in Section~\ref{sec:algo} and applied to both simulated and experimental data in Section~\ref{sec:numerics}. Finally, the conclusions are drawn in Section~\ref{sec:conclusion}.

\section{Complete electrode model and its parametric extension}
\label{sec:CEM}

This section introduces an extension of the CEM, allowing the use of parameter-dependent conductivities, contact resistances, electrode positions and boundary shapes. For a justification of the standard CEM, see~\cite{Cheng89,Somersalo92}. We work in two spatial dimensions and with $M \in \N \setminus \{1\}$ electrodes of the same known width, but the generalization to three dimensions and/or to the case of electrodes with unknown shapes is conceptually straightforward. 

\subsection{Parametrization of the measurement set-up}
\label{sec:parametr}
Let $N = N_{\sigma} + N_{\gamma} + 2M$ denote the number of parameters living in the hypercube
$$
\Upsilon = \Upsilon_\sigma \times \Upsilon_\gamma \times \Upsilon_E  \times \Upsilon_z = [-1/2,1/2]^N. 
$$
We decompose $y = (y_\sigma, y_\gamma, y_E, y_z) \in \Upsilon$, where the subvectors $y_\sigma \in \Upsilon_\sigma \subset \R^{N_\sigma}$, $y_\gamma \in \Upsilon_\gamma\subset \R^{N_\gamma}$, $y_E \in \Upsilon_E \subset\R^{M}$ and $y_z \in \Upsilon_z \subset \R^{M}$ correspond to the parametrizations of the conductivity field, the boundary curve, the electrode positions and the contact resistances, respectively.

Let us first introduce a parametrization for the boundary curve, that is, a continuous map
\begin{equation}
\label{para_bound}
\Upsilon_\gamma \ni y_\gamma \mapsto \gamma(\, \cdot \, , y_\gamma) \in C^{0,1}_L\big(\R; \R^2\big),
\end{equation}
where, for every $y_\gamma \in \Upsilon_\gamma$, $\gamma(\, \cdot \, , y_\gamma): [0,L) \to \R^2$ defines a bounded, closed, non-self-intersecting, Lipschitz curve parametrized in the counterclockwise direction. Here and in what follows, the subscript $L>0$ indicates that the elements of the considered function space are $L$-periodic. The domain enclosed by
$$
\Gamma(y_\gamma) := \big\{ \gamma(\phi, y_\gamma) \ | \ \phi \in [0,L) \big\} 
$$
is denoted $D(y_\gamma)$. We assume there exists a natural bi-Lipschitz homeomorphism
\begin{equation}
\label{homeo}
\Phi(\, \cdot \, , y_\gamma): D(y_\gamma) \to D(0)
\end{equation}
for all $y_\gamma \in \Upsilon_\gamma$. In our numerical tests, $D(0)$ is an origin-centered open disk and the domains $D(y_\gamma)$, $y_\gamma \in \Upsilon_\gamma$, are star-shaped with respect to the origin, meaning that one can define the mappings $\Phi(\, \cdot \,, y_\gamma)$ by suitably scaling the distance to the origin; see Section~\ref{sec:algo} for the details. 

Given a parametrization for the domain boundary, the position of an electrode is determined by a curve parameter corresponding to its starting point. To be more precise, after introducing a suitable mapping
$$
\Upsilon_E \ni y_E \mapsto \theta(y_E) \in [0,L)^M,
$$
the electrodes are parametrized by the set-valued functions,
\begin{equation}
\label{para_elec}
\Upsilon_\gamma \times \Upsilon_E \ni (y_\gamma, y_E) \mapsto
E_m(y_\gamma, y_E) \subset \Gamma(y_\gamma), \qquad m = 1, \dots, M,
\end{equation}
where 
$$
E_m(y_\gamma, y_E) := \big\{ \gamma(\phi,y_\gamma) \in \Gamma(y_\gamma) \ | \ 0 < {\rm dist} \big(\gamma(\theta_m(y_E),y_\gamma), \gamma(\phi,y_\gamma)\big) < \omega \big\} 
$$
with ${\rm dist}(x,z)$ denoting the distance between the points $x,z \in \Gamma(y_\gamma)$ along $\Gamma(y_\gamma)$ in the counterclockwise direction and $\omega$ being the known width of the electrodes. The mapping $\theta: \Upsilon_E \to [0,L)^M$ is assumed to be continuous, when $L$ is identified with $0$ on the image side, and to satisfy the condition
\begin{equation}
\label{nooverlap}
\min_{j \not= k} \min_{y_E \in \Upsilon_E} \min_{y_\gamma \in \Upsilon_\gamma} {\rm dist}\big(\gamma(\theta_j( y_E),y_\gamma), \gamma(\theta_k(y_E),y_\gamma)\big) > \omega
\end{equation}
which guarantees that the electrodes do not overlap, or change their order. (In our numerical tests, the parametrization is slightly simpler as the $m$th starting parameter $\theta_m$ depends only on the corresponding component of $y_E$.)

The conductivity field is parametrized by first introducing the dependence on $y_\sigma$ in the `unperturbed' reference domain $D(0)$ with the help of a continuous mapping
$$
\Upsilon_\sigma \ni y_\sigma \mapsto \sigma_0(\, \cdot \, , y_\sigma) \in L_+^\infty (D(0)):= 
\{ \kappa \in L^\infty(D(0)) \ | \ {\rm ess} \inf \kappa > 0 \},
$$
and then defining the actual domain-dependent parametrization via
\begin{equation}
\label{para_cond}
\Upsilon_\sigma \times \Upsilon_\gamma \ni (y_\sigma,y_\gamma) \mapsto \sigma(\, \cdot \, , y_\sigma,  y_\gamma) := \sigma_0\big(\Phi(\, \cdot \, , y_\gamma) ,y_\sigma \big) \in L_+^\infty (D(y_\gamma)).
\end{equation}
Finally, the contact resistances $z \in \R_+^M$ are parametrized simply by a continuous map
\begin{equation}
\label{para_res}
\Upsilon_z \ni y_z \mapsto z(y_z) \in \R_+^M,
\end{equation}
where, in fact, $z_m$, $m=1, \dots, M$, only depends on the corresponding component of~$y_z$.

In what follows, we often write $\gamma^y = \gamma(\, \cdot \, , y_\gamma)$, $\Gamma^y = \Gamma(y_\gamma)$, $D^y = D(y_\gamma)$, $\Phi^y = \Phi(\, \cdot \, , y_\gamma)$, $E^y_{m} = E_m(y_\gamma, y_E)$, $\sigma^y =  \sigma(\, \cdot \, , y_\sigma,  y_\gamma)$ and $z^y = z(y_z)$ to simplify the notation.

\subsection{Parameter-dependent CEM}
Assume that the parametrizations \eqref{para_bound}, \eqref{para_elec}, \eqref{para_cond} and \eqref{para_res} are given, denote by $\R^{M}_\diamond$ the mean-free subspace of $\R^{M}$, and let $I \in \R^{M}_\diamond$ define the net current feeds through the electrodes. According to the CEM~\cite{Cheng89}, for a fixed parameter vector $y \in \Upsilon$, the electromagnetic potential $u^y$ inside $D^y$ and the potentials $U^y \in \R^M$ on the electrodes satisfy the elliptic boundary value problem
\begin{align}
\label{pCEM}
\begin{array}{ll}
\nabla \cdot \big(\sigma^y \nabla u^y \big) = 0 \qquad  &\text{in} \ D^y, \\[8pt] 
{\displaystyle \frac{\partial u^y}{\partial \nu}} = 0 \qquad &\text{on} \ \Gamma^y\setminus \overline{E^y},\\[2mm] 
{\displaystyle u^y+z^y_m \sigma^y \frac{\partial u^y}{\partial \nu}}= U^y_m \qquad &\text{on} \ E^y_m, \quad m=1, \dots, M, \\[3mm] 
{\displaystyle \int_{E^y_m} \sigma^y \frac{\partial u^y}{\partial \nu} \, \dd S} = I_m, \qquad & m=1,\ldots,M, 
\end{array}
\end{align}
where $\nu = \nu(x)$ denotes the exterior unit normal of $\Gamma^y$ and $E^y = \cup_{m=1}^M E^y_m$. It follows immediately from the material in \cite{Somersalo92} and the properties of the parametrizations introduced in Section~\ref{sec:parametr} that \eqref{pCEM} has a unique solution $(u^y, U^y) \in (H^1(D^y) \oplus \R^M)/ \R =: \mathcal{H}^y$ for all $y \in \Upsilon$. Moreover, one can write a relatively explicit $y$-independent estimate for the $\mathcal{H}^y$-norm of $(u^y, U^y)$ as revealed by the following analysis.

The variational formulation of \eqref{pCEM} is to find $(u^y, U^y) \in \mathcal{H}^y$ such that~\cite{Somersalo92}
\begin{equation}
\label{pCEMvar}
B^y\big((u^y,U^y), (v,V) \big) \, = \, I \cdot V \qquad \textrm{for all } (v,V) \in \mathcal{H}^y, 
\end{equation}
where the bilinear form $B^y: \mathcal{H}^y \times \mathcal{H}^y \to \R$ is defined as
$$
B^y\big((w,W), (v,V) \big) \, = \, \int_{D^y} \sigma^y \nabla w \cdot \nabla v \, \dd x + \sum_{m=1}^M \frac{1}{z^y_m} \int_{E^y_m} (W_m-w) (V_m-v) \, \dd S .
$$
Let us define
\begin{equation}
\label{sigmaminmax}
\varsigma_{-} = \min_{y_\sigma \in \Upsilon_\sigma} {\rm ess} \inf \sigma_0(\,\cdot \,, y_\sigma), \qquad \varsigma_{+} = \max_{y_\sigma \in \Upsilon_\sigma} \| \sigma_0(\, \cdot \, , y_\sigma) \|_{L^{\infty}(D(0))}, 
\end{equation}
and
\begin{equation}
\label{zminmax}
\zeta_{-} = \min_{m } \min_{y_z\in \Upsilon_z} z_m(y_z), \qquad 
\zeta_{+} = \max_{m} \max_{y_z\in \Upsilon_z} z_m(y_z)  .
\end{equation}
Furthermore, let $C_{y, {\rm tr}} > 0$ be the norm of the trace operator
$$
{\rm tr}: v \mapsto v|_{\Gamma^y}, \ \  H^1(D^y) \to L^2(\Gamma^y),
$$
and $C_{y, {\rm P}}> 0$ be the Poincar\'{e}--Wirtinger constant for $D^y$, that is, the smallest constant such that
$$
\| v - \bar{v} \|_{H^1(D^y)} \leq C_{y, {\rm P}} \, \|\nabla v \|_{L^2(D^y)} \qquad \textrm{for all } v \in H^1(D^y),
$$
where $\bar{v} \in \R$ denotes the mean of $v$ over $D^y$. Finally, set
$$
C_{\rm tr} = \sup_{y \in \Upsilon} C_{y, {\rm tr}}, \qquad C_{\rm P} = \sup_{y \in \Upsilon} C_{y, {\rm P}} .
$$
Note that the dependence of the trace norm on the corresponding domain is an active research topic; see,~e.g.,~\cite{Rossi08} and the references therein. On the other hand, consult~\cite{Boulkhemair07} for a result that could be applied to  the Poincar\'{e}--Wirtinger constant in our setting. In the following, we simply assume that both $C_{\rm tr}$ and $C_{{\rm P}}$ are finite. 

\begin{theorem}
The bilinear form $B^y:  \mathcal{H}^y \times \mathcal{H}^y \to \R$ is uniformly bounded and coercive, that is,
$$
B^y\big((w,W), (v,V) \big) \, \leq \, \max \left\{ \varsigma_+ + \frac{2 C_{\rm tr}^2}{\zeta_-}, \frac{2 \omega}{\zeta_-} \right\} \|(w,W) \|_{\mathcal{H}^y}  \|(v,V) \|_{\mathcal{H}^y}   
$$
and
$$
B^y \big((v,V), (v,V) \big) \, \geq \, \left( \max \left\{ \frac{C_{\rm P}^2}{\varsigma_{-}} \Big(1 + \frac{2C_{\rm tr}^2}{\omega}\Big) , \frac{2 \zeta_+}{\omega} \right\} \right)^{-1} \|(v,V) \|_{\mathcal{H}^y}^2 
$$
for all $y \in \Upsilon$. 
\end{theorem}

\begin{proof}
The result follows by keeping track of the constants in \cite[Proof of Lemma~2.5]{Hyvonen04} and accounting for the slight difference  between the $H$-norm employed in \cite{Hyvonen04} and the natural norm of $\mathcal{H}^y$, i.e.,
$$
\|(v,V)\|_{\mathcal{H}^y}^2 \, := \, \inf_{c \in \R} \left( \| v - c \|_{H^1(D^y)}^2 + | V - c \, \mathbf{1} |^2 \right) \,
$$
where $\mathbf{1} = (1, \dots, 1) \in \R^M$. \quad 
\end{proof}

\begin{corollary}
The solution of \eqref{pCEM} satisfies the uniform bound
\begin{equation}
\label{estimate}
\| (u^y, U^y) \|_{\mathcal{H}^y} \, \leq \, \max \left\{ \frac{C_{\rm P}^2}{\varsigma_{-}}\Big(1 + \frac{2C_{\rm tr}^2}{\omega}\Big) , \frac{2 \zeta_+}{\omega} \right\} | I |
\end{equation}
for all $y \in \Upsilon$.
\end{corollary}

\begin{proof}
The claim is a direct consequence of the Lax--Milgram lemma. \quad
\end{proof}

For the convergence of (standard) FEM, it is essential to have control over the behavior of the higher Sobolev norms of $u^y \in H^1(D^y)/ \R$. To this end, denote by $C_{y,\epsilon} > 0$ the norm of the zero continuation operator from $H^{1/2 -\epsilon}(E^y)$ to $H^{1/2 -\epsilon}(\Gamma^y)$, $0 < \epsilon < 1$, and by $\tilde{C}_{y,\epsilon}>0$ the norm of the solution mapping
$$
H^{1/2 - \epsilon}(\Gamma^y) \ni f^y \mapsto v^y \in H^{2 - \epsilon}(D^y)/ \R,
$$
corresponding to the Neumann problem
$$
\nabla \cdot (\sigma^y \nabla v^y) = 0 \quad {\rm in} \ D^y, \qquad
\sigma^y  \frac{\partial v^y}{\partial \nu}  = f^y \quad {\rm on} \  \Gamma^y.
$$
Moreover, let $\hat{C}_{y}$ be the norm of the Neumann-to-Dirichlet map
$$
L^2_\diamond(\Gamma^y) \ni f^y \mapsto v^y|_{\Gamma_y} \in H^{1}(\Gamma^y)/ \R,
$$
where $L^2_\diamond(\Gamma^y)$ is the mean-free subspace of $L^2(\Gamma^y)$.
Finally, set
$$
C_{\epsilon} = \sup_{y \in \Upsilon} C_{y, \epsilon}, \qquad \tilde{C}_{\rm \epsilon} = \sup_{y \in \Upsilon} \tilde{C}_{y, \epsilon}, \qquad \hat{C} = \sup_{y \in \Upsilon} \hat{C}_{y}.
$$
It is once again assumed that the parametrization of our measurement setting is regular enough to make these definitions unambiguous as well as $C_{\epsilon}$, $\tilde{C}_{\rm \epsilon}$ and $\hat{C}$ finite for the considered $0 < \epsilon < 1$ (cf.~\cite{Adams03, Lions72}).

\begin{corollary}
Let $C_1>0$ be the constant on the right-hand side of \eqref{estimate}. 
For any $0 < \epsilon < 1$, the first part of the solution to \eqref{pCEM} satisfies the uniform bound
\begin{equation}
\label{second_estimate}
\| u^y \|_{H^{2-\epsilon}(D^y) / \R} \, \leq \, 
\frac{\sqrt{2} C_\epsilon \tilde{C}_\epsilon}{\zeta_-}  \Big( \frac{\hat{C}}{\zeta_-} + 1 \Big)\max\{ C_{\rm tr}, \sqrt{\omega}\} \, C_1 |I|  
\end{equation}
for all $y \in \Upsilon$.
\end{corollary}

\begin{proof}
By definition,
$$
\| u^y \|_{H^{2-\epsilon}(D^y)/ \R}  \leq  \, \tilde{C}_{\epsilon} \left\| \sigma^y \frac{\partial u^y}{\partial \nu} \right\|_{H^{1/2 - \epsilon}(\Gamma^y)}  
\leq \, C_{\epsilon} \tilde{C}_{\epsilon} \left\| \sigma^y \frac{\partial u^y}{\partial \nu} \right\|_{H^{1}(E^y)},
$$ 
where we also used the trivial embedding $H^{1}(E^y) \subset H^{1/2 - \epsilon}(E^y)$ to deduce the second inequality.
Now the claim follows by carefully keeping track of the constants in \cite[Proof of Lemma 3.1]{Hanke11b} and \cite[Proof of Lemma 2.1]{Hyvonen09}. \qquad
\end{proof}

Observe that the constants appearing on the right-hand side of \eqref{second_estimate} are not independent of each other: For example, the norm of the zero continuation $C_\epsilon$ certainly depends on $\omega>0$ and obviously $\hat{C}$, $\tilde{C}_\epsilon$ and $C_{\rm tr}$ are intimately connected. Moreover, the estimate \eqref{second_estimate} is not optimal; as an example, consult \cite{Costabel96,Darde16} for more careful analysis of the dependence on $\zeta_-$. Be that as it may, \eqref{second_estimate} arguably gives a general idea of how the parametrization of the  measurement configuration affects the bound on the $H^{2-\epsilon}(D^y)$-norm of the interior potential.

\begin{remark}
Although the estimate on the $H^{2-\epsilon}(D^y)$-norm of the electromagnetic potential \eqref{second_estimate} is connected to the accuracy of the numerical forward solution at the chosen sparse grid points over $\Upsilon$ (cf.~Section~\ref{sec:cFEM}), from the standpoint of efficient polynomial collocation it would be more important to prove analytic dependence of the solution pair $(u^y, U^y)$ on the parameter vector $y \in \Upsilon$; see,~e.g.,~\cite{Babuska10}. However, such investigations are left for future studies.
\end{remark}

In the following, we systematically choose the ground level of potential by identifying $(H^1(\Omega) \oplus \R^M)/ \R \simeq H^1(\Omega) \oplus \R^M_\diamond$.

\section{cFEM applied to the CEM}
\label{sec:cFEM}
In this section we describe how the parameter-dependent CEM forward problem is discretized in both spatial and parametric dimensions.
Let $\mathcal{I} \in \R^{M \times (M-1)}$ denote a current matrix whose columns form an arbitrary but fixed basis for the space of feasible net current feeds, that is, for $\R^M_\diamond$.
The corresponding numerical solutions $\mathcal{U}^y \in \R^{M \times (M-1)}$ for the {\em electrode} potentials in problem \eqref{pCEM} (or in its variational formulation \eqref{pCEMvar}) with a fixed $y \in \Upsilon$ can be computed by using standard FE techniques; recall that $y = (y_\sigma, y_\gamma, y_E, y_z)$ defines the conductivity, the contact resistances and the geometric set-up for the forward problem~\eqref{pCEM}.
We continue to assume that the ground level of potential is chosen such that each column of $\mathcal{U}^y$ has zero mean.
It is straightforward to show that, due to the linear dependence on the current pattern $I$ in \eqref{pCEMvar}, the solutions $\widetilde{\mathcal{U}}^y$ corresponding to another current matrix $\widetilde{\mathcal{I}}$ satisfy $\widetilde{\mathcal{U}}^y = \mathcal{U}^y \widetilde{\mathcal{I}}^\dagger \mathcal{I}$, where $(\cdot)^\dagger$ denotes the Moore--Penrose pseudoinverse.
Thus, we pay no attention to choosing the current feeds in what follows.

By requiring the variational formulation \eqref{pCEMvar} to hold for all $M-1$ current patterns and for all FEM basis functions, we end up with a matrix equation
\begin{equation}
\label{pCEMmat}
A^y v^y = F^y,
\end{equation}
where the stiffness matrix $A^y$ depends on all parameters $y \in \Upsilon$ and $F^y$, having $M-1$ columns, depends only on the subvectors $y_\gamma$ and $y_E$ (through the meshing of $D^y$).
Ultimately, we are interested only in those $Q := M(M-1)$ elements of the unknown matrix $v^y$ that define the electrode potentials $\mathcal{U}^y$.

The aim of the cFEM  is to construct an explicit parameter dependence into the numerical solution.
More precisely, we seek for a polynomial mapping
$$
\Upsilon \ni y \mapsto \mathcal{U}(y) \in \R^Q
$$
so that $\mathcal{U}(y) \approx \mathcal{U}^y$ element-wise.
For notational convenience, we actually consider $\mathcal{U}(y)$ as a vector that is obtained by stacking the matrix columns on top of each other.
We write the numerical parametric solution in the form
\begin{equation}
\label{paramU}
[\mathcal{U}(y)]_q = \sum_{p=1}^P \widehat{U}_{q,p} \mathscr{L}_p(y), \qquad q = 1, \ldots, Q,
\end{equation}
where $\widehat{U}_{q,p} \in \R$ and $\mathscr{L}_p$ are suitably scaled Legendre polynomials in $N$ variables (cf.~\cite{Gautschi04}).
It remains to choose the actual set of polynomials and also to determine the coefficient matrix $\widehat{U} \in \R^{Q \times P}$.

We henceforth assume that the polynomial basis is normalized as
\[
\int_{\Upsilon}\mathscr{L}_p(y)\mathscr{L}_{p'}(y)\,{\rm d}y=\delta_{p,p'},\qquad p,p' = 1, \dots, P,
\]
where $\delta_{p,p'}$ denotes the Kronecker symbol defined to be unity whenever its indices coincide and vanishing otherwise. Performing a discrete projection of the parametric solution \eqref{paramU} onto the tensorized Legendre polynomial basis $(\mathscr{L}_p)_{p=1}^P$ in $\Upsilon$ yields a representation of the coefficients given by the integrals
\begin{equation*}
\widehat U_{q,p}=\int_{\Upsilon}[\mathcal{U}(y)]_q\mathscr{L}_p(y)\,{\rm d}y,\qquad q=1, \dots ,Q,  \ \  p=1, \dots ,P,
\end{equation*}
which we approximate by a sparse grid quadrature based on nested Clenshaw--Curtis rules. The sparse grid method was first introduced in \cite{Smolyak63} and comprehensive analyses of its approximation properties were developed later in \cite{NovRit96,NovRit99,WasilWoz95}. The application of sparse grid quadratures to collocation methods was pioneered in the context of parametric partial differential equations in such works as \cite{Babuska10,Xiu06}. In the following, we give a brief overview of these techniques applied to the problem considered in this paper.

The sparse grid method is based on extending a family of univariate quadrature rules into the high-dimensional parametric region $\Upsilon\subset\mathbb{R}^N$ by considering a sparsity-promoting linear combination of tensorized collections of univariate quadrature operators. The nested Clenshaw--Curtis rules in the interval $[-1/2,1/2]$ are based on the sequence $m(1)=1$ and $m(n)=2^{n-1}+1$ for $n>1$, which corresponds to the abscissae $y_1^{(1)}=0$ and
\[
y_k^{(n)}=-\frac{1}{2}\cos\left(\frac{(k-1)\pi}{m(n)-1}\right),\qquad k=1,\dots,m(n),\ n\in\mathbb{N}\setminus\{1\}.
\]
The abscissae characterize a sequence of positive weights $(w_k^{(n)})_{k=1}^{m(n)}$ that define the Clenshaw--Curtis quadrature rules
\[
Q_nf=\sum_{k=1}^{m(n)}w_k^{(n)}f(y_k^{(n)})\approx\int_{-1/2}^{1/2}f(y)\,{\rm d}y,
\]
which are exact for all polynomials of degree not more than $m(n)$. The tensor products of these quadrature operators are defined as
\[
\bigotimes_{k=1}^NQ_{\alpha_k}f=\sum_{i_1=1}^{m(\alpha_1)}\cdots\sum_{i_N=1}^{m(\alpha_N)}w_{i_1}^{(\alpha_1)}\cdots w_{i_N}^{(\alpha_N)}f(y_{i_1}^{(\alpha_1)},\dots,y_{i_N}^{(\alpha_N)}),
\]
where $\alpha_k\in\mathbb{N}$ for $1\leq k\leq N$.

A particular case of sparse grid quadrature is the well known Smolyak's construction \cite{WasilWoz95}. The $N$-dimensional Smolyak rule of order $K\geq0$ based on the nested Clenshaw--Curtis rules is given by
\[
\mathcal{Q}_{N,K}=\sum_{\max\{N,K+1\}\leq |\alpha|\leq N+K}(-1)^{N+K-|\alpha|}\binom{N-1}{N+K-|\alpha|}\bigotimes_{k=1}^NQ_{\alpha_k},
\]
where $\alpha=(\alpha_1,\dots,\alpha_N)\in\N^N$ and $|\alpha|=\alpha_1+\ldots+\alpha_N$. The function evaluations are carried out in the \emph{sparse grid}
\[
\Theta_{N,K}=\bigcup_{|\alpha|=N+K}\Theta_{\alpha_1}\times\cdots\times\Theta_{\alpha_N},
\]
where $\Theta_n=\{y_k^{(n)}\}_{k=1}^{m(n)}$. The sparse grid $\Theta_{N,K}$ has the asymptotic cardinality 
$$
n_{N,K} := \, \#\Theta_{N,K}\sim\frac{2^K}{K!}N^K
$$ 
as $N$ tends to infinity for a fixed $K$ \cite{NovRit99}. By tabulating the collocation nodes $(y_k^{(N,K)})_{k=1}^{n_{N,K}}$ in $\Theta_{N,K}$ and their respective weights $(w_k^{(N,K)})_{k=1}^{n_{N,K}}$, the Smolyak quadrature rule can be rewritten as a cubature rule
\[
\mathcal{Q}_{N,K}f=\sum_{k=1}^{n_{N,K}}w_k^{(N,K)}f(y_k^{(N,K)}).
\]
The Smolyak rule generalizes the polynomial exactness of the underlying univariate rules~\cite{NovRit96}. Let $\Pi_K^N$ denote the space of all polynomials in $N$ variables of total degree at most $K$. Then
$$
\mathcal{Q}_{N,K}f=\int_{\Upsilon}f(y)\,{\rm d}y
$$
for all multivariate polynomials $f$~such that
\[
f\in\sum_{|\alpha|=N+K}(\Pi_{m(\alpha_1)}^1\otimes\cdots\otimes\Pi_{m(\alpha_N)}^1).
\]
In particular, it can be shown that the rule is exact for all $f\in\Pi_{2K+1}^N$ whenever $K<3N$~\cite{NovRit99}. For $K\geq 3N$, the related total degree space is different and we omit it.

The coefficients $\widehat U_{q,p}\in\mathbb{R}$ that appear in the operator \eqref{paramU} can now be approximated by using the cubature rule 
\begin{equation}
\label{smolyakcubature}
\widehat U_{q,p} \approx \mathcal{Q}_{N,K}([\mathcal{U}(\,\cdot \,)]_q\mathscr{L}_p)=\sum_{k=1}^{n_{N,K}}w_k^{(N,K)}[\mathcal{U}(y_k^{(N,K)})]_q\mathscr{L}_p(y_k^{(N,K)}),
\end{equation}
where the needed nodal evaluations of $\mathcal{U}(\,\cdot \,)$ are replaced by those of the FEM solutions $\mathcal{U}^{(\cdot)}$.
The error accumulation of the collocated solution depends on the error introduced in the numerical solution of the CEM for fixed realizations of the parameter $y\in\Upsilon$, the truncation error that stems from the representation \eqref{paramU} and the aliasing error caused by the cubature rule in \eqref{smolyakcubature}.

\section{Implementation}
\label{sec:algo}
As emphasized in,~e.g.,~\cite{Hyvonen15}, an EIT inversion algorithm based on stochastic or parametric FEM consists of two distinct parts.
In the \emph{pre-measurement} processing, the explicit parameter dependence \eqref{paramU} is constructed by using cFEM.
Unlike in \cite{Hyvonen15}, we do not assume that the geometry of the measurement setting is known during the pre-measurement processing, but instead include parameters for the boundary curve and electrode positions in the cFEM problem.

In the \emph{post-measurement} processing, the parametric solution is fitted to the measurement data with respect to the parameter vector $y$.
This part is often treated as a least squares minimization problem, which involves either a Tikhonov functional or a {\em maximum a posteriori} (MAP) estimator in a Bayesian approach.
Once a minimizing vector $y \in \Upsilon$ is found, recovering the quantities of interests follows straightforwardly by considering the mappings introduced in the pre-measurement step.

\subsection{Pre-measurement processing}
\label{sec:premeas}
This subsection introduces one possible set of concrete mappings that were abstractly given in Section~\ref{sec:parametr}.
That is, we consider functions that map the parameter vectors $y_\gamma$, $y_E$, $y_\sigma$ and $y_z$ to boundary curve, electrode positions, conductivity field and contact resistances, respectively.
We will frequently use the fact that the components of the parameter vectors lie in the interval $[-1/2, 1/2]$.
The mappings introduced here are certainly not the only feasible ones.

The boundary curve $\Gamma(y_\gamma)$ is represented as a perturbed circle, where the amount of perturbation is determined by a linear combination of $N_\gamma \geq 3$ quadratic B-splines (cf.,~e.g.,~\cite{Hollig03}).
To this end, we choose $L = 2 \pi$ in \eqref{para_bound} and write
$$
\gamma(\phi,y_\gamma) = \big (r(\phi,y_\gamma), \phi \big)
$$
in polar coordinates.
We then choose the maximum radial perturbations
$$
\rho_{-} = \min_{\phi \in [0,2\pi]} \min_{y_\gamma \in \Upsilon_\gamma} r(\phi,y_\gamma), \qquad
\rho_{+} = \max_{\phi \in [0,2\pi]} \max_{y_\gamma \in \Upsilon_\gamma} r(\phi,y_\gamma)
$$
and set $\rho_0 = (\rho_- + \rho_+)/2$.
The radial coordinate for $\gamma$ can now be written as
\begin{equation}
\label{radial_coor}
r(\phi, y_\gamma) = \rho_0 + \sum_{i=1}^{N_\gamma} (\rho_+ - \rho_-) [y_\gamma]_i \psi_i(\phi),
\end{equation}
where $\psi_i \in C^1_{2 \pi}(\R)$ are nonnegative and uniform quadratic B-splines that form a partition of unity.
The unperturbed case $y_\gamma = 0$ corresponds to a circle with radius $\rho_0$, i.e., $D(0)$ is a disk of radius $\rho_0$.
Each spline satisfies $\lvert \mathrm{supp}(\psi_i) \cap [0,2 \pi] \rvert = 6 \pi / N_\gamma$.
Thus, the deformations are local, as illustrated in Figure~\ref{fig:domains}.
We define the homeomorphism \eqref{homeo} as
\begin{equation}
\label{homeodef}
\Phi \big((r',\phi), y_\gamma \big) = \bigg(\frac{\rho_0}{r(\phi,y_\gamma)} r', \phi \bigg),
\end{equation}
which holds whenever $(r', \phi) \in D(y_\gamma)$.

For all $y_E \in \Upsilon_E$, we define the starting angles of the electrodes as
\begin{equation}
\label{start_angle}
\theta_m(y_E) = (m-1) \frac{2 \pi}{M} + 2 \alpha [y_E]_{m}, \qquad m = 1, \ldots, M.
\end{equation}
Here, the offset parameter $\alpha \geq 0$ is sufficiently small so that the non-overlapping condition \eqref{nooverlap} is satisfied.
(Actually, the existence of such $\alpha$ also requires that $\rho_-$ is sufficiently large compared to $\omega$,~i.e.,~$2\pi \rho_- > M\omega$, but this is assumed to be true.)
Because in EIT the absolute orientation of the imaged object in space cannot be determined, we can as well fix one of the starting angles and decrease the number of parameters by one.
For simplicity, however, we keep $N$ as defined and fix the starting angle of the first electrode by (re-)defining $\theta_1(y_E) = 0$ for all $y_E \in \Upsilon_E$.

\begin{figure}
\center{
{\includegraphics[height=1.8in]{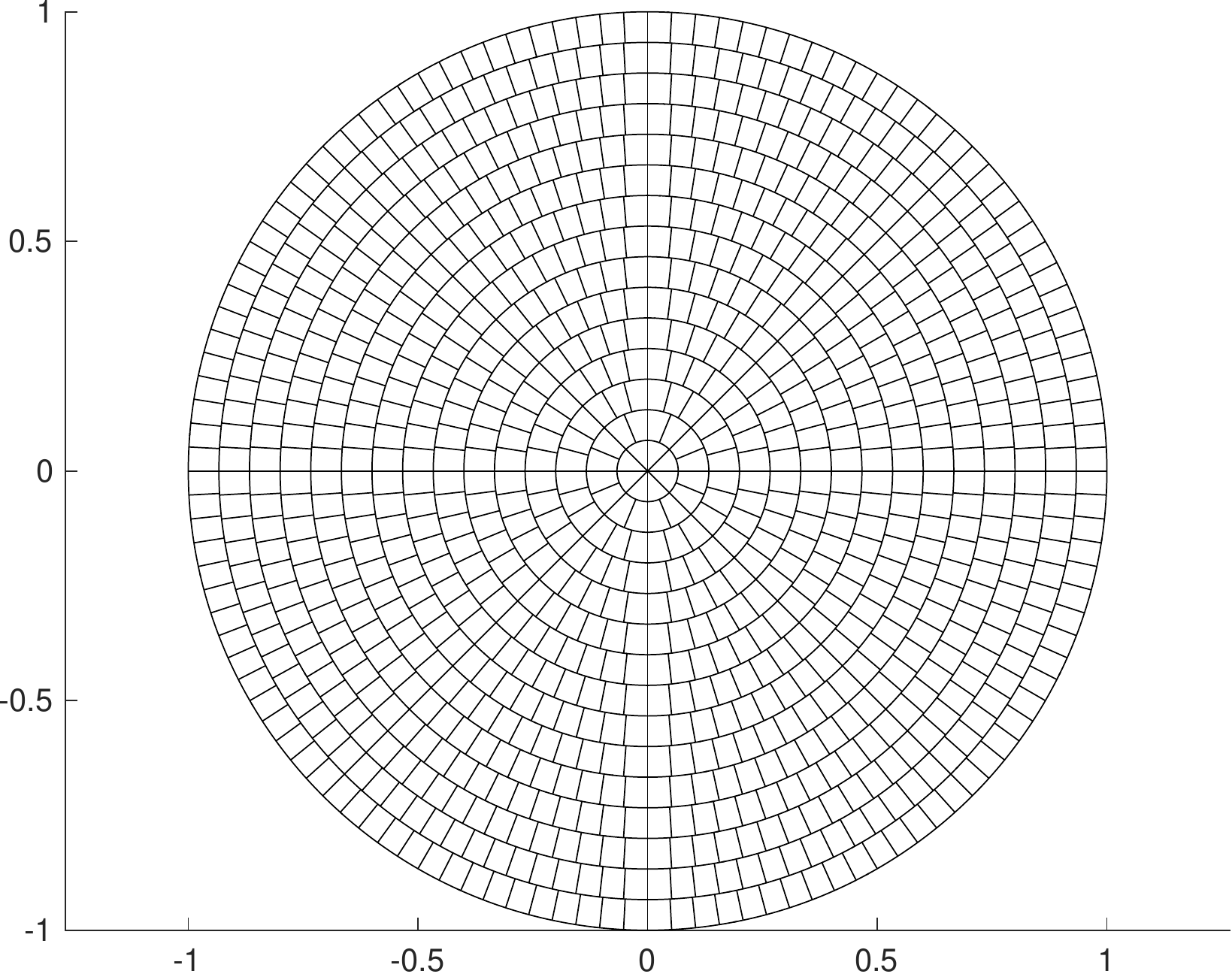}}
\qquad
{\includegraphics[height=1.8in]{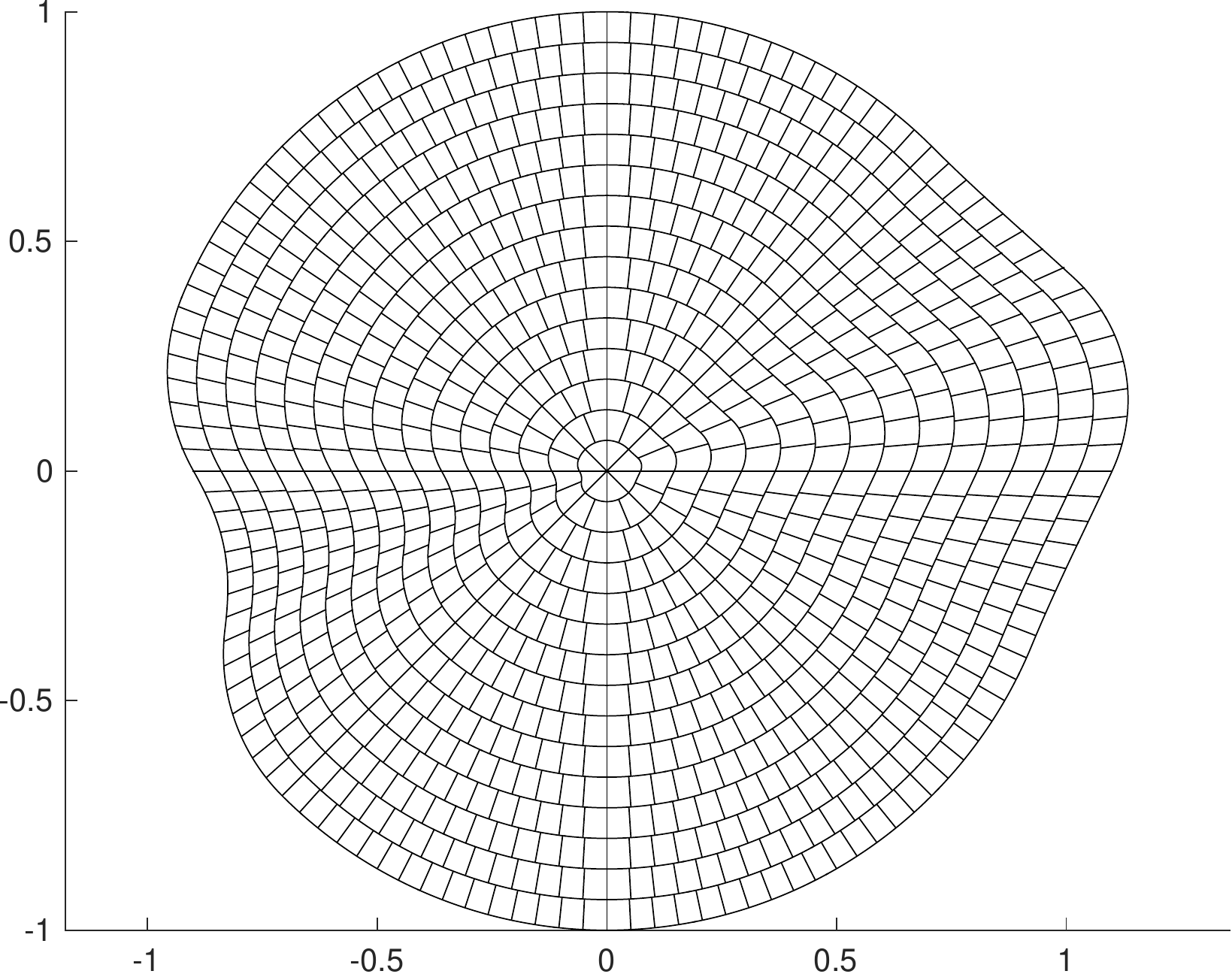}}
}
\caption{The discretization of the conductivity in a disk (left) and in a domain that is obtained by maximally perturbing 2 out of $N_\gamma = 16$ spline coefficients defining the boundary curve (right). In these pictures $N_\sigma = 960$, $\rho_-=0.8$ and $\rho_+=1.2$.}
\label{fig:domains}
\end{figure}

We resort to a piecewise constant representation for the conductivity. 
Other possible choices include Karhunen--Lo\`eve eigenfunctions corresponding to a (prior) random field and suitable FEM basis functions.
We partition the canonical domain $D(0)$ into $N_\sigma$ pairwise disjoint subdomains that satisfy
$$
\bigcup_{i=1}^{N_\sigma} \overline{\chi_i^{-1}(1)} = \overline{D(0)},
$$
where $\chi_i$ is the indicator function of the $i$th subdomain.
An example partitioning is shown in Figure~\ref{fig:domains}.
For $x \in D(0)$, the canonical conductivity $\sigma_0$ is defined as
\begin{equation}
\label{parsigma}
\sigma_0(x,y_\sigma) = \sum_{i=1}^{N_\sigma} \chi_i(x) \exp \mathopen{}\bigg( \frac{1}{2} \log(\varsigma_- \varsigma_+) + \log\bigg(\frac{\varsigma_+}{\varsigma_-}\bigg) [y_\sigma]_i \bigg)\mathclose{}
\end{equation}
for some \emph{given} values $0 < \varsigma_- \leq \varsigma_+$.
It is easy to see that \eqref{parsigma} satisfies \eqref{sigmaminmax}.
The domain-dependent parametrization of the conductivity follows from \eqref{para_cond} and \eqref{homeodef}.
The contact resistances are defined according to
\begin{equation}
\label{parz}
z_m(y_z) = \exp \mathopen{}\bigg( \frac{1}{2} \log(\zeta_- \zeta_+) + \log \mathopen{}\bigg(\frac{\zeta_+}{\zeta_-}\bigg) [y_z]_m \mathclose{}\bigg)\mathclose{}, \qquad m = 1, \ldots, M,
\end{equation}
which agrees with \eqref{zminmax}.

Typically, the number of parameters defining the conductivity field is much higher than the number of boundary curve parameters.
In Figure \ref{fig:domains}, for example, we have chosen $N_\sigma = 960$, whereas $N_\gamma = 16$.
On the other hand, the number of parameters defining the contact resistances and electrode positions is directly determined by the number of electrodes, which is usually quite low.
Due to the structure of a typical sparse collocation grid, the number of different computational domains and finite element meshes is therefore moderate compared to the total number of collocation points.
In fact, most of the forward problems in our numerical examples are computed in the canonical domain $D(0)$ with equiangled electrodes and with contact resistances set to $z(0)$.
These problems merely correspond to perturbing a few values in the stiffness matrix $A^y$ in \eqref{pCEMmat}.

Besides being easily parallelizable, there may also be a lot of symmetries that can be used to reduce the amount of computation in the pre-measurement phase.
For example, if the conductivity is discretized as in Figure~\ref{fig:domains}, the number of electrodes is divisible by eight and the current matrix involves certain symmetries, then solving a forward problem corresponding to a conductivity node (i.e., a collocation node where only $y_\sigma$ contains nonzero values) simultaneously solves several problems where both current feed and the conductivity field are rotated.
Only re-ordering of the resulting potential values is required.
Similar symmetries may arise, e.g., when the boundary curve is perturbed but other components are fixed.
Exploiting these symmetries is beyond the scope of this article.

\subsection{Post-measurement processing}
\label{sec:postmeas}
The aim of the post-measurement processing is to fit the parametric solution \eqref{paramU} to measurement data; once again, recall that $y \in \Upsilon$ appearing in \eqref{paramU} parametrizes the conductivity, the contact resistances and the geometric set-up for \eqref{pCEM}.
Let $\mathcal{V} \in \R^Q$ denote a vector of noisy potential measurements corresponding to some known set of current feeds.
As mentioned in the beginning of Section \ref{sec:cFEM}, the actual current values can be arbitrary as long as they are known and they form a proper basis.
Moreover, the linearity with respect to the applied current pattern implies that uncertainties in the current feeds can be assumed to be propagated to the measurement errors of potentials.
In this paper, we tackle the inverse problem of EIT by considering a nonlinear least squares problem of the form
\begin{equation}
\label{leastsquares}
\min_{y \in \Upsilon} \big\{ \lvert \mathcal{U}(y) - \mathcal{V} \rvert^2 + \lambda^2 \lvert R(y) \rvert^2 \big\},
\end{equation}
where $\lambda \geq 0$ is a regularization parameter and $R \colon \R^N \to \R^{N'}$ is a differentiable regularization operator for an arbitrary $N' \in \N$. Notice that the connection between a (local) minimizer of \eqref{leastsquares} and (regularized) solutions of the underlying undiscretized inverse problem of EIT in the framework of the CEM is nontrivial to analyze; see,~e.g.,~\cite{Hiptmair15,Schwab12} for related considerations. 

We refer to \cite{Nocedal99} for discussion about nonlinear least squares algorithms.
In short, most algorithms are based on successive linearizations and require evaluating both $\mathcal{U}(y)$ and $R(y)$ as well as their Jacobian matrices for different values of $y \in \Upsilon$.
The reconstructions in Section~\ref{sec:numerics} are obtained by using the \texttt{lsqnonlin} function in MATLAB with a user-supplied Jacobian.
It is shown in \cite{Mustonen15} that the cost of evaluating $\mathcal{U}(y)$ and its Jacobian is $O(QN^k)$, where $k$ is the largest polynomial (total) degree in the chosen $P$-dimensional polynomial subspace.
On the other hand, solving a linearized least squares subproblem typically has a complexity of $O(QN^2)$, assuming that $N'\lesssim N$. 

As the regularization operator we use a block-diagonal matrix $R \in \R^{N \times N}$ containing blocks $R_\sigma \in \R^{N_\sigma \times N_\sigma}$, $R_\gamma \in \R^{N_\gamma \times N_\gamma}$, $R_E \in \R^{M \times M}$ and $R_z \in \R^{M \times M}$.
The conductivity block $R_\sigma$ is defined via its inverse Cholesky factor, that is,
$$
[R_\sigma^{-\mathrm{T}} R_\sigma^{-1}]_{i,j} = \kappa_{\sigma}^2 \exp \mathopen{}\bigg( - \frac{\lvert x_i - x_j \rvert^2}{2 \beta^2} \bigg) + \varepsilon \delta_{i,j}, 
\qquad i,j = 1, \dots, N_\sigma,
$$
where $\kappa_{\sigma}, \beta, \varepsilon > 0$ are free parameters to be specified by the operator of the algorithm and $x_i \in D(0)$ is the polar mean of the $i$th subdomain of $D(0)$,~i.e.,~the point defined by the mean values of the polar coordinates in that subdomain.
Loosely speaking, this corresponds to the assumption that the conductivity is {\em a priori} a log-normal random field with variance-like parameter $\kappa_{\sigma}^2$ and correlation length $\beta$;  the role of the {\em small} parameter $\varepsilon$ is just to guarantee the invertibility of the matrix.
We could as well write $R_\sigma = R_\sigma(y_\gamma)$ and compute the distances in perturbed domains, but this would cause extra work with insignificant effect on the reconstruction.

The (Cholesky factors of the) regularization matrices $R_\gamma$, $R_E$ and $R_z$ are diagonal, i.e., 
$$
R_\gamma = \kappa_{\gamma}^{-1} \mathrm{I}, \qquad R_E = \kappa_{E}^{-1} \mathrm{I}, \qquad R_z = \kappa_{z}^{-1} \mathrm{I},
$$
where $\mathrm{I}$ denotes an identity matrix of the appropriate size and $\kappa_{\gamma}, \kappa_{E}, \kappa_{z} > 0$ are regularization parameters. If \eqref{leastsquares} were considered as computation of a MAP estimate within the Bayesian paradigm,  the positive parameters $\kappa_{\gamma}, \kappa_{E}$ and $\kappa_{z}$ would act as the standard deviations of the (independent) zero-mean Gaussian priors for the components of $y_\gamma$, $y_E$ and $y_z$, respectively (cf.,~e.g.,~\cite{Kaipio04}). In particular, under the Bayesian interpretation, the prior for the contact resistances is log-normal (cf.~\eqref{parz}) and those for the electrode angles and the coefficients of the spline-like boundary perturbations Gaussian.

\section{Numerical experiments}
\label{sec:numerics}
We demonstrate the feasibility of the proposed method by numerical experiments in two spatial dimensions. First, the parametric solution $\mathcal{U}(y)$ is constructed as explained in Sections \ref{sec:cFEM} and \ref{sec:premeas}. The conductivity is discretized with $N_\sigma = 960$ parameters as in Figure~\ref{fig:domains}. For the boundary curve, we choose $N_\gamma = 16$ splines and the number of electrodes is $M = 16$. Thus, the total number of parameters is $N = 1008$. By using the notation of Section~\ref{sec:premeas}, we choose the minimum and maximum radii as $(\rho_-, \rho_+) = (15,20)$. The maximum offset for an electrode angle is $\alpha = 0.1$ and the width of the electrodes is $\omega = 2$.  For the conductivity and the contact resistances we choose $(\varsigma_-, \varsigma_+) = (0.1, 1)$ and $(\zeta_-, \zeta_+) = (0.05, 1)$, respectively. (In the tests based on experimental data, the units of length, conductivity and contact resistance are cm, ${\rm mS}/{\rm cm}$ and ${\rm k} \Omega \, {\rm cm}^2$, respectively.)

The tensorized Legendre polynomial basis $(\mathscr{L}_p)_{p=1}^P$ is chosen such that it spans the space containing all bilinear, linear and constant polynomials in $N$ variables.
This results in $P = (N^2+N)/2+1 = 508\,537$ and the complexity of evaluating $\mathcal{U}(y)$ and its Jacobian matrix becomes $O(QN^2)$.
The Smolyak rule of order $K=2$ based on the nested Clenshaw--Curtis rules is used in the computation of the coefficients $\widehat U_{q,p}$. This rule is exact for integrands in $\Pi_5^N$, i.e., for all $N$-variate polynomials of total degree at most $5$, resulting in $2\,034\,145$ collocation nodes in~\eqref{smolyakcubature}. The corresponding CEM forward problems are solved by a standard FEM with about $2000$ piecewise linear basis functions and appropriate refinement of the employed meshes close to  the electrodes. Recall that these forward problems can be solved in parallel. With our hardware (53GB RAM, Intel Xeon X5650 CPU) and non-optimized MATLAB implementation, this pre-measurement phase of forming \eqref{paramU} took a few hours. In what follows, we employ the same parametric forward solution in all reconstructions, except for the fixed-geometry reconstructions (see Figures~\ref{fig:naive} and \ref{fig:tank2naive}), which for comparison are computed by setting $\rho_- = \rho_+$ and $\alpha=0$ in the pre-measurement phase.

As mentioned in Section~\ref{sec:postmeas}, the inverse problem is also solved with MATLAB. We actually treat the problem \eqref{leastsquares} as an \emph{unconstrained} minimization problem and solve it by {\tt lsqnonlin} with zero initial vector\footnote{This initial guess corresponds to a disk of radius $\rho_0=17.5$, equiangled electrodes, a homogeneous conductivity $\sigma \equiv \sqrt{\varsigma_- \varsigma_+} \approx 0.32$ and contact resistances $z_m = \sqrt{\zeta_- \zeta_+} \approx 0.22$, $m=1,\dots, 16$.} and the {\tt levenberg-marquardt} option, because this is simple, easily reproducible and there is no reason to expect that some other technique would result in a significantly more accurate localization of a minimizer for \eqref{leastsquares}. Apart from $\lambda > 0$, we use the same values for the free parameters of the post-measurement processing in all numerical tests, namely $\beta = 4$, $\varepsilon = 10^{-4}$ and $\kappa_{\sigma} = \kappa_\gamma = \kappa_E = \kappa_z = 0.25$ (cf.~Section \ref{sec:postmeas}). Making these parameters case-specific would certainly improve the reconstructions to a certain extent, but it would also conflict our aim of demonstrating that a generic set of parameter values leads to good reconstructions both with simulated data and for different experimental settings. Recall that within the Bayesian paradigm $\kappa_{\sigma}$,  $\kappa_\gamma$, $\kappa_E$ and $\kappa_z$ can be interpreted as the prior standard deviations for the components of $y_{\sigma}$, $y_\gamma$, $y_E$ and $y_z$, respectively, meaning that the sizes of the former indicate the amount of variation one expects in the latter {\em a priori}. Combining this observation with the parametrizations \eqref{radial_coor} and \eqref{start_angle}--\eqref{parz} indicates how much fluctuation is expected in the boundary curve, the electrode positions, the pixel values of the conductivity and the contact resistances, respectively, prior to the measurements. Moreover, the choice of the correlation length $\beta>0$ is related to the anticipated characteristic length of conductivity variations inside the imaged object.

Unless otherwise stated, the post-measurement processing phase lasted only a few seconds on a modern desktop computer.

\subsection{Simulated data}

\begin{figure}
\center{
{\includegraphics[height=2in]{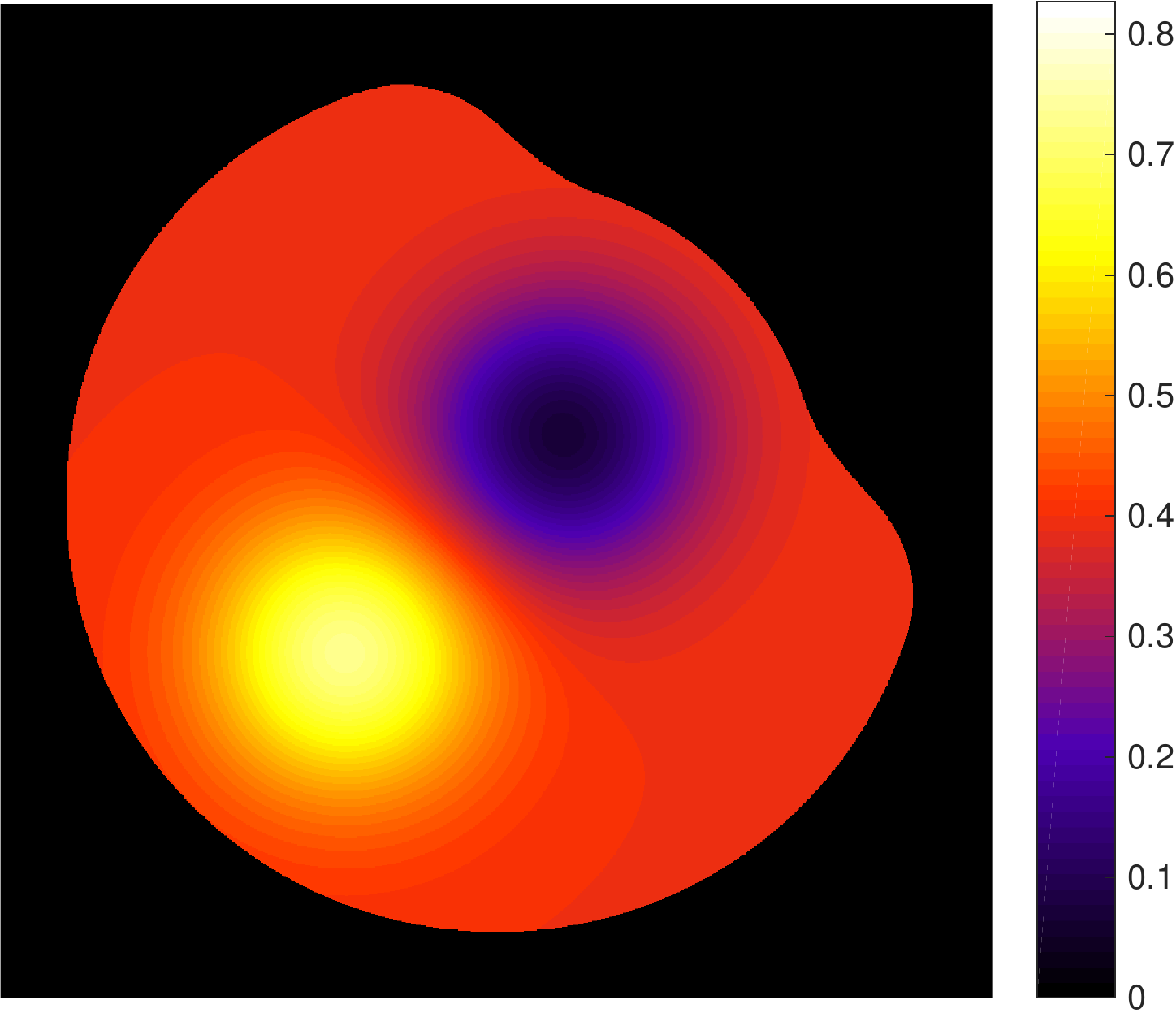}}
\qquad
{\includegraphics[height=2in]{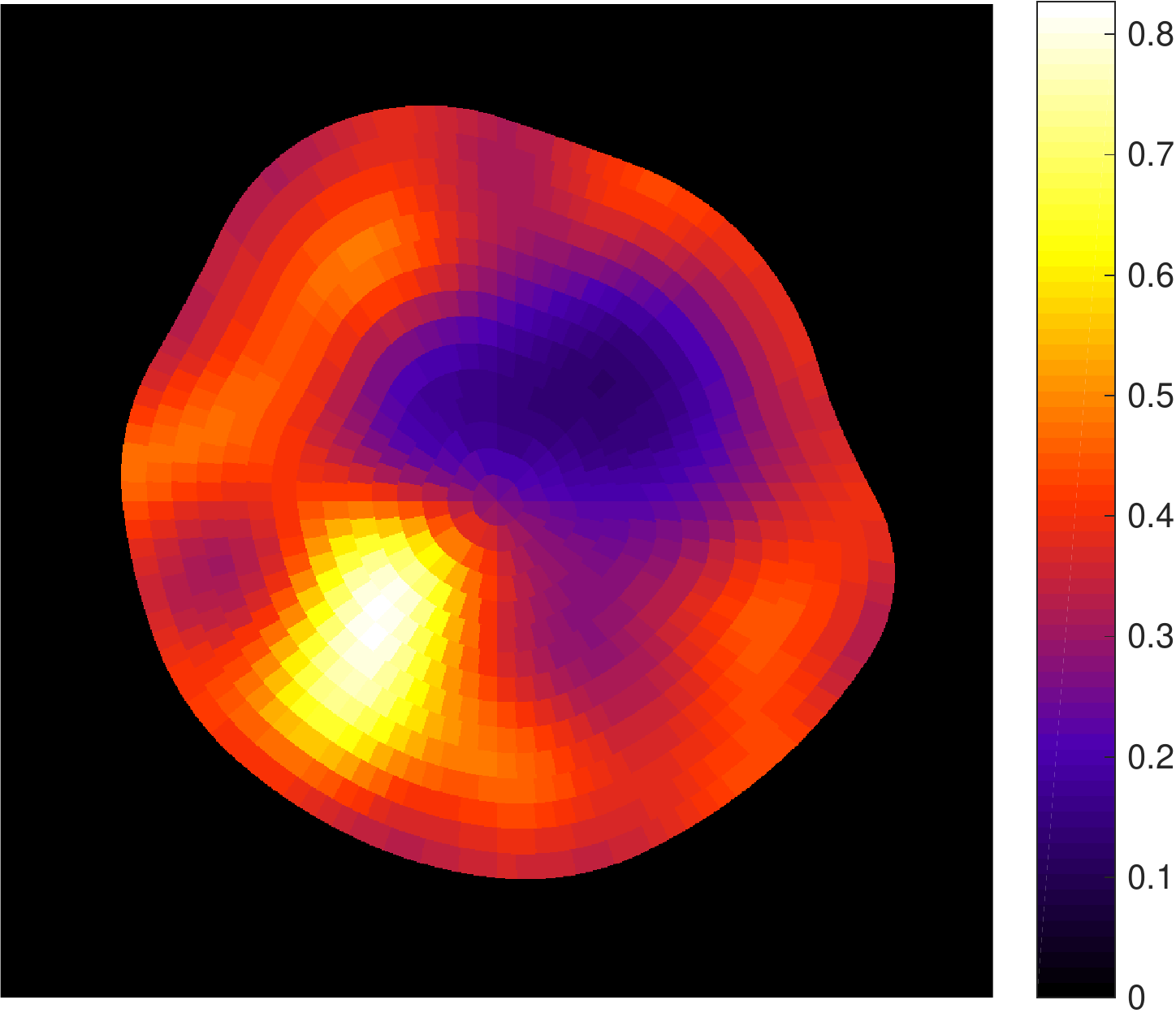}} 
} \\
\hspace{3mm}
\center{
{\includegraphics[height=2in]{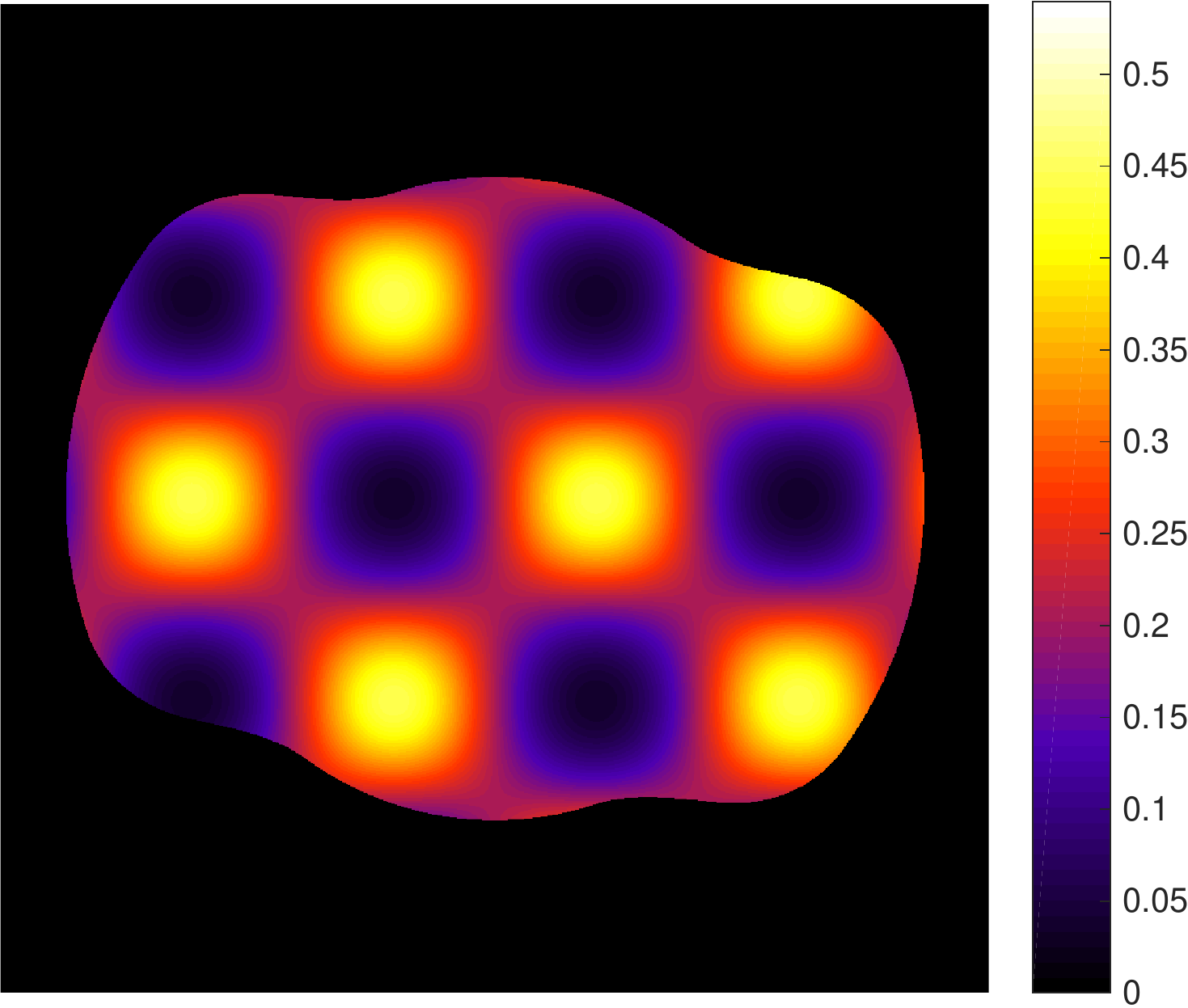}}
\qquad
{\includegraphics[height=2in]{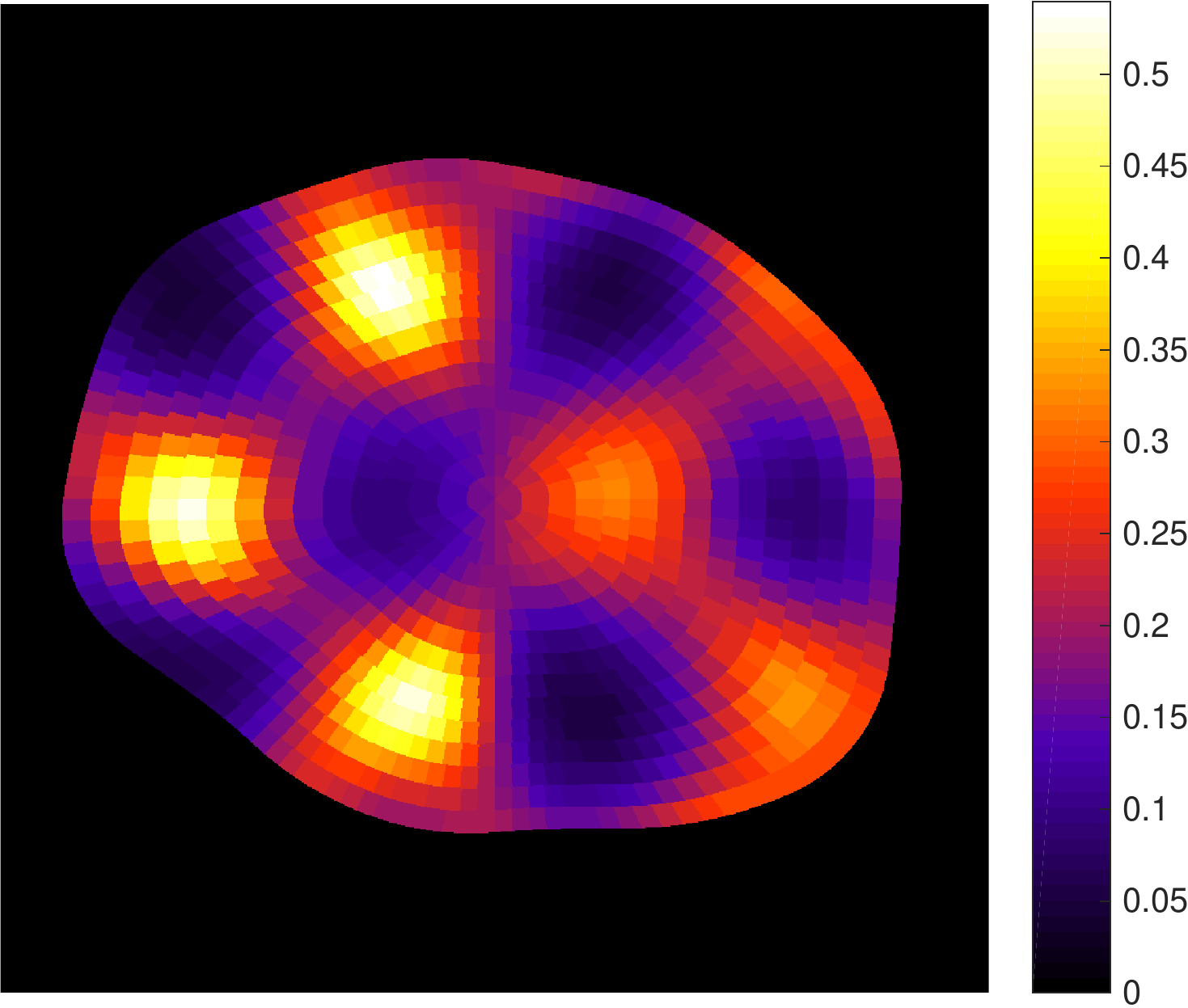}}
}
\caption{Simulated data. Left: the target configurations. Right: the reconstructions.}
\label{fig:simulated1}
\end{figure}

Let us first consider simulated measurements. The considered artificial target conductivities and body shapes are shown in the left-hand column of Figure~\ref{fig:simulated1}. Neither of the two boundary curves can be exactly represented by the parametrization \eqref{radial_coor} with $N_\gamma = 16$. For both phantoms, there are sixteen electrodes of two units width distributed somewhat evenly along the respective boundary curve. The corresponding contact resistances are similar random perturbations; see Figure~\ref{fig:simulated2} for the details. The employed current patterns,~i.e.,~the columns of $\mathcal{I}$, are
\begin{equation}
\label{current_input}
I^m =  {\mathrm e}_1 - {\mathrm e}_{m+1}  , \qquad m = 1, \dots, M-1,
\end{equation}
with ${\mathrm e}_m$ denoting the $m$th Euclidean basis vector of $\R^M$. The electrode measurements are simulated by first solving the necessary CEM forward problems by a FEM with piecewise linear basis functions --- on considerably denser meshes than the ones employed for the inverse solver --- to obtain the `exact' potential vector $\mathcal{U}^{\rm exct} \in \R^{Q}$.
Subsequently, the actual data are formed as
$$
\mathcal{V} \, = \, \mathcal{U}^{\rm exct} + \eta,
$$
where the components of $\eta \in \R^Q$ are independent realizations of a normally distributed random variable with zero mean and standard deviation
$$
\tau \, = \, 10^{-3}\! \max_{j,k=1, \dots, Q} \big| \mathcal{U}_j^{\rm exct} - \mathcal{U}_k^{\rm exct} \big|, 
$$
which corresponds to $0.1$\% of noise. 
The reconstruction algorithm is then applied to $\mathcal{V}$ with the regularization parameter $\lambda = 2\tau$. If \eqref{leastsquares} is interpreted as the determination of a MAP estimate within the Bayesian paradigm, then $\lambda$ plays the role of the standard deviation for the assumed zero mean Gaussian noise process with independent components (cf.,~e.g.,~\cite{Kaipio04}). In particular, we assume here twice as high noise level than actually contaminating the (simulated) measurements.

\begin{figure}
\center{
{\includegraphics[height=1.8in]{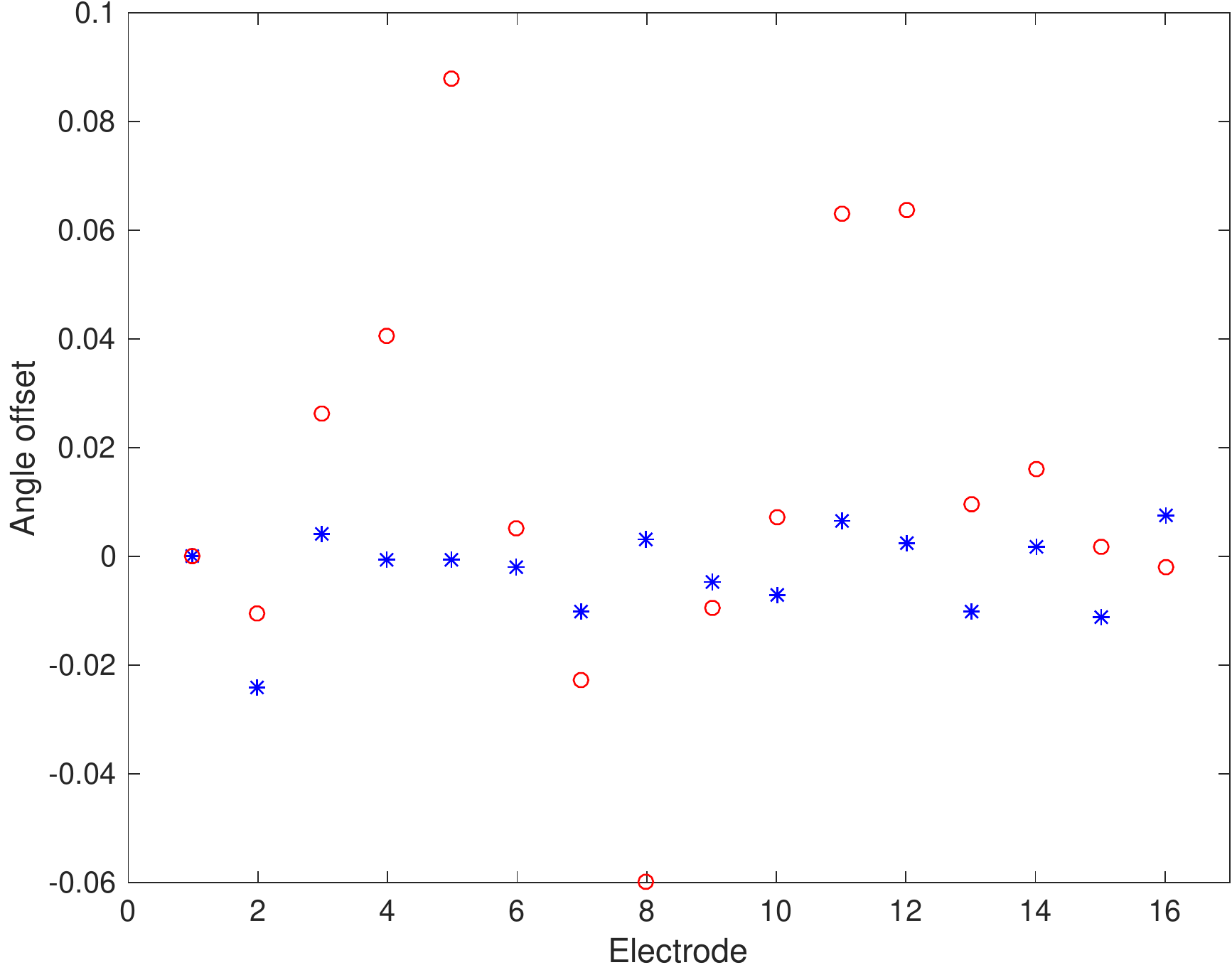}}
\qquad
{\includegraphics[height=1.8in]{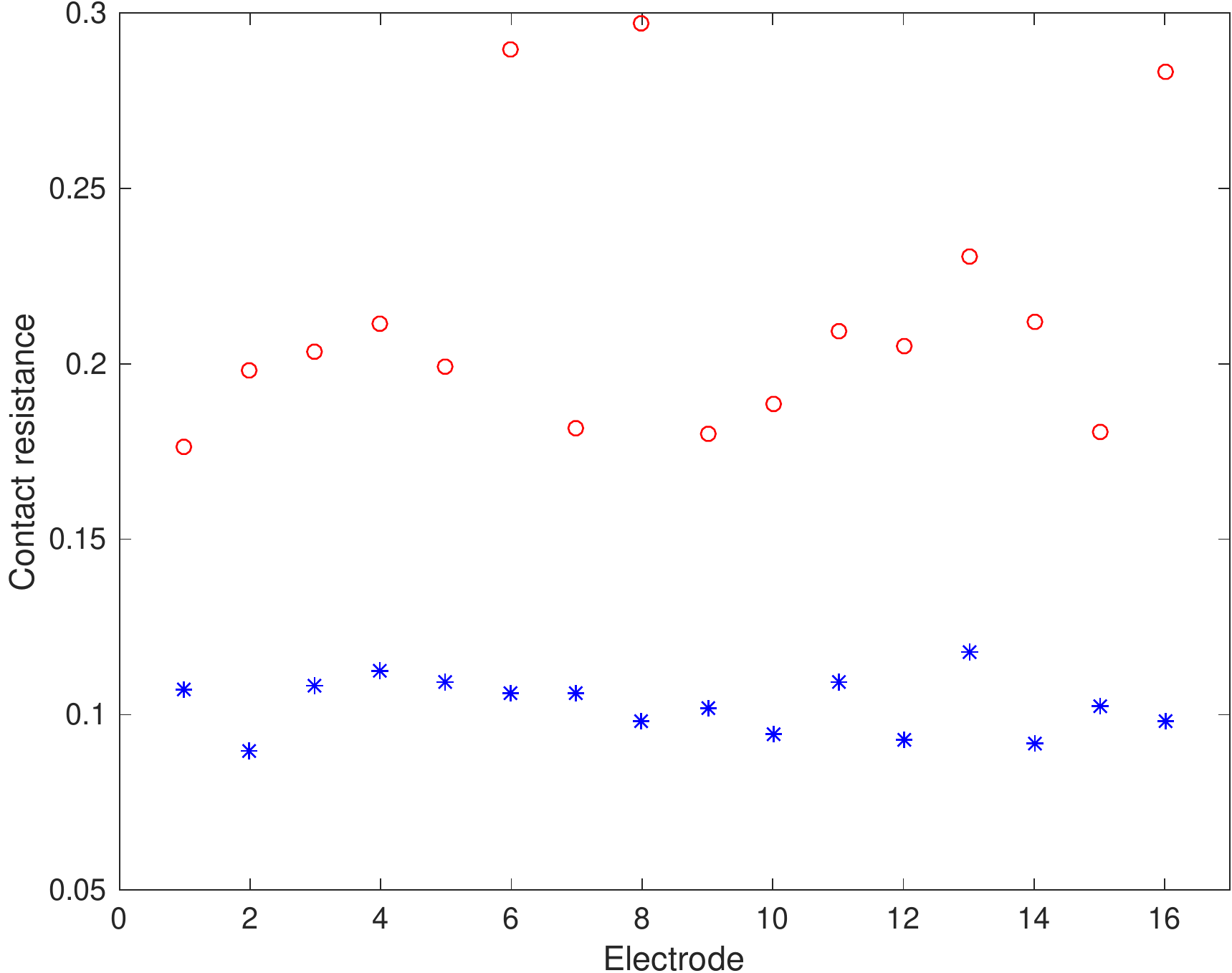}} 
} \\
\hspace{3mm}
\center{
{\includegraphics[height=1.8in]{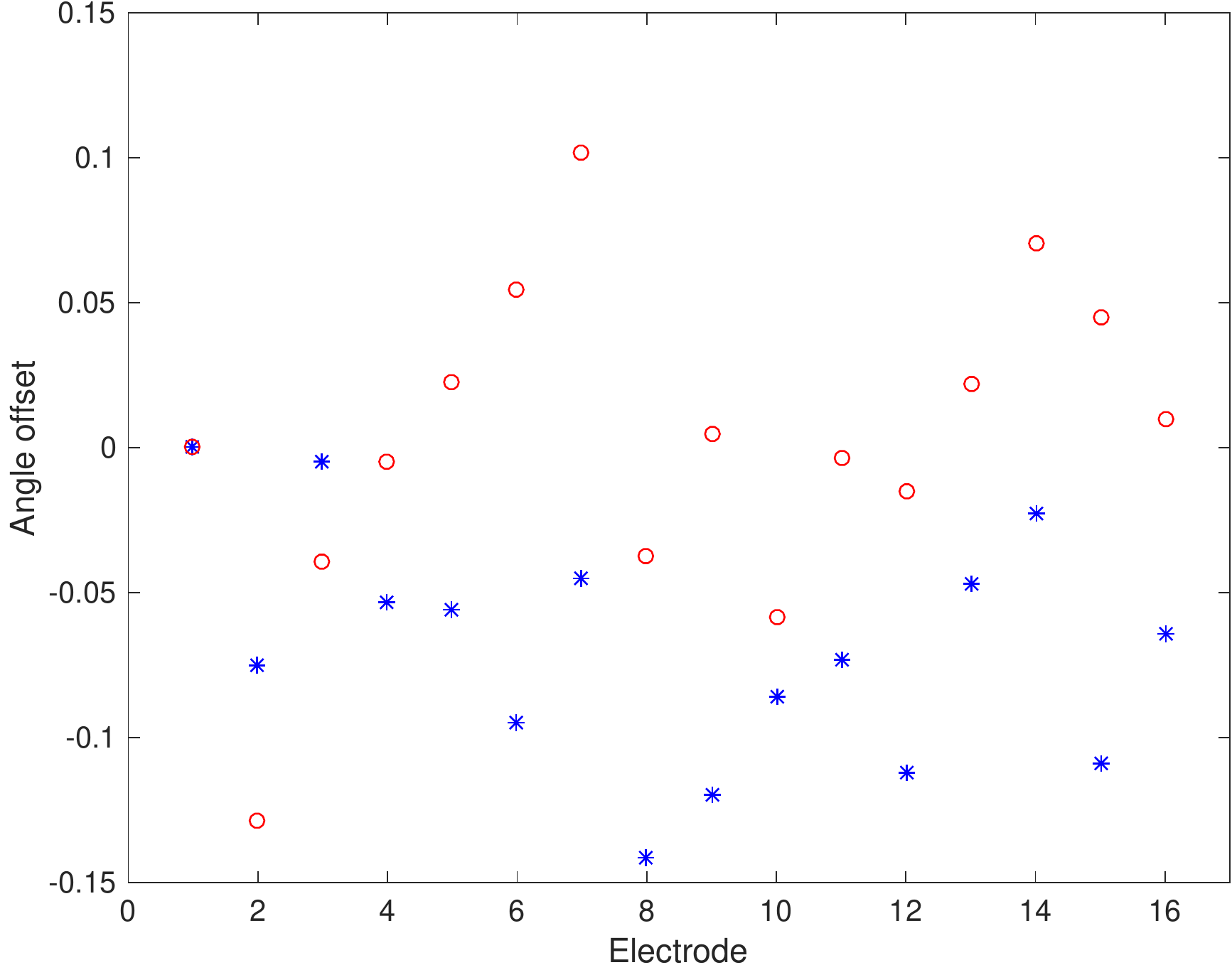}}
\qquad
{\includegraphics[height=1.8in]{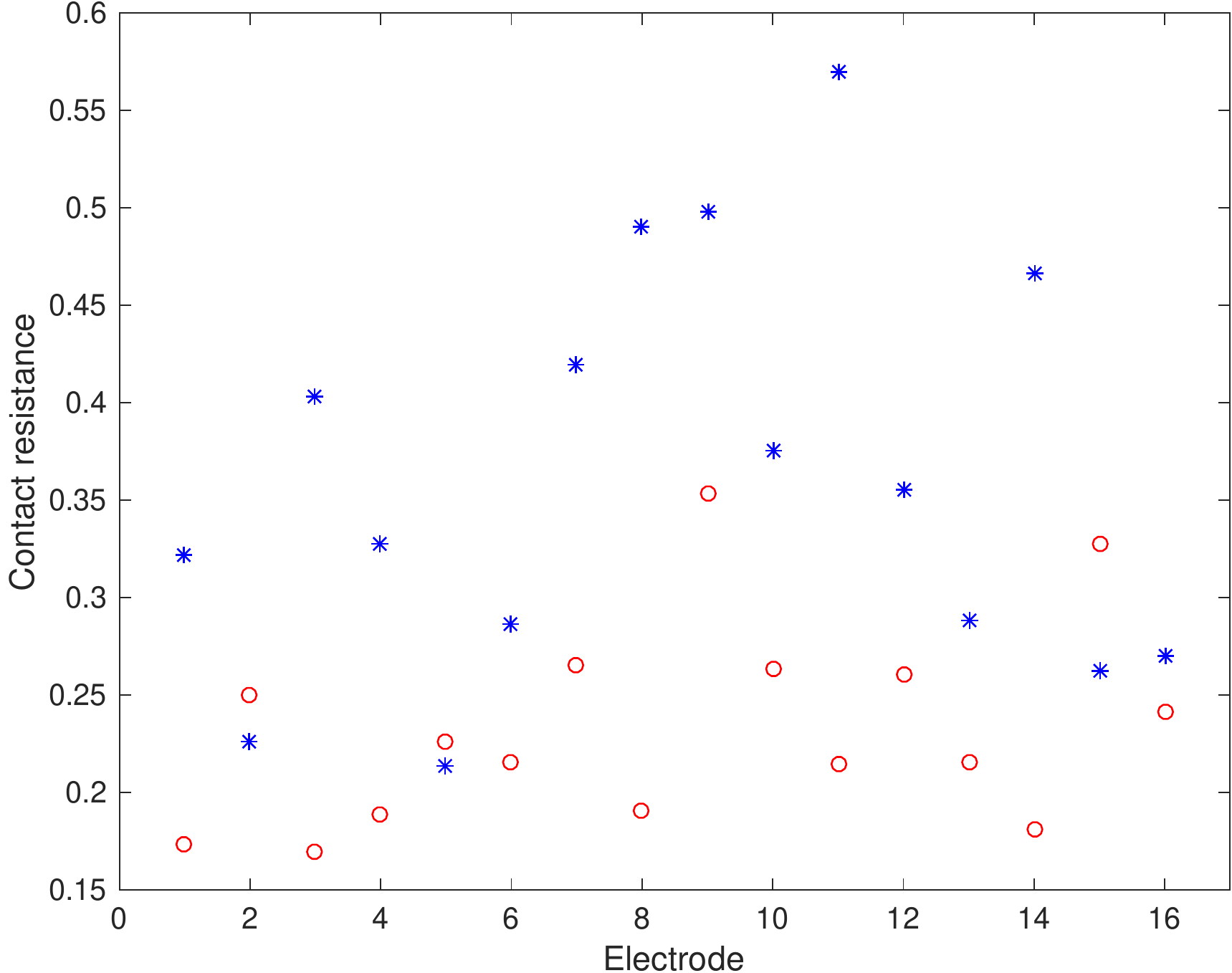}}
}
\caption{Simulated data. Left: the reconstructed electrode angles (red circles) compared with the true ones (blue asterisks). Right: the reconstructed contact resistances (red circles) compared with the true ones (blue asterisks). The top row corresponds to the top row of Figure~\ref{fig:simulated1} and the bottom row to the bottom row of Figure~\ref{fig:simulated1}}
\label{fig:simulated2}
\end{figure}

The resulting conductivity reconstructions are illustrated in the right-hand column of Figure~\ref{fig:simulated1}. They demonstrate that the algorithm is capable of capturing the qualitative behavior of the conductivity phantoms as well as the exterior boundary shapes of the examined objects.
The circumference of the top left object in Figure~\ref{fig:simulated1} is approximately $121$, while the circumference of its reconstruction is only $111$.
The corresponding values for the second object are $115$ and $113$, respectively. Thus, the reconstructed circumferences are close to that of $D(0)$, i.e., $2\pi \cdot 17.5 \approx 110$. 

The reconstructed contact resistances and electrode angles are compared with the true ones in Figure~\ref{fig:simulated2}; it is obvious that the algorithm does not estimate these quantities accurately. The inaccuracy demonstrated by Figure~\ref{fig:simulated2} is probably mainly due to the nontrivial interplay between the object shape, electrode angles and contact resistances: Some features of the data caused by the object shape may be less `expensive' to explain by varying the electrode angles and/or contact resistances under the chosen regularization/prior model. Such behavior is most evident in the top right image of Figure~\ref{fig:simulated2}, where the too high values for the reconstructed contact resistances arguably compensate for the too small size of the reconstructed object in the top right image of Figure~\ref{fig:simulated1}. It should also be noted that the orientation of the reconstruction in space is intimately connected to the reconstructed electrode angles: All (random) angle offsets for the target in the bottom left image of Figure~\ref{fig:simulated1} are negative, which leads to a reconstruction that is slightly rotated in the clockwise direction. This eliminates the systematic bias in the true electrode angles (in comparison to the parameter value $y_E = 0$) and results in reconstructed angle offsets that take both positive and negative values.

To conclude this section, let us demonstrate what happens if the uncertainties related to the measurement geometry are simply ignored. Figure~\ref{fig:naive} shows the reconstructions of the target configurations in Figure~\ref{fig:simulated1} produced by our reconstruction algorithm when the pre-measurement phase is computed in a disk of radius $\rho_0 = 17.5$ with equally spaced electrodes attached to its boundary.
As illustrated by the extremely poor reconstructions of the conductivity in Figure~\ref{fig:naive}, this naive approach is intolerable, which is inline with the findings of~\cite{Barber88, Breckon88, Darde13a, Darde13b, Kolehmainen97}.
Moreover, the minimization algorithm converges slowly: As an example, the left image in Figure~\ref{fig:naive} required $121$ function evaluations, whereas the top right image in Figure~\ref{fig:simulated1} was obtained with only $12$ function evaluations, although the same (\texttt{lsqnonlin}'s default) stopping criterion was used.

\begin{figure}
\center{
{\includegraphics[height=2in]{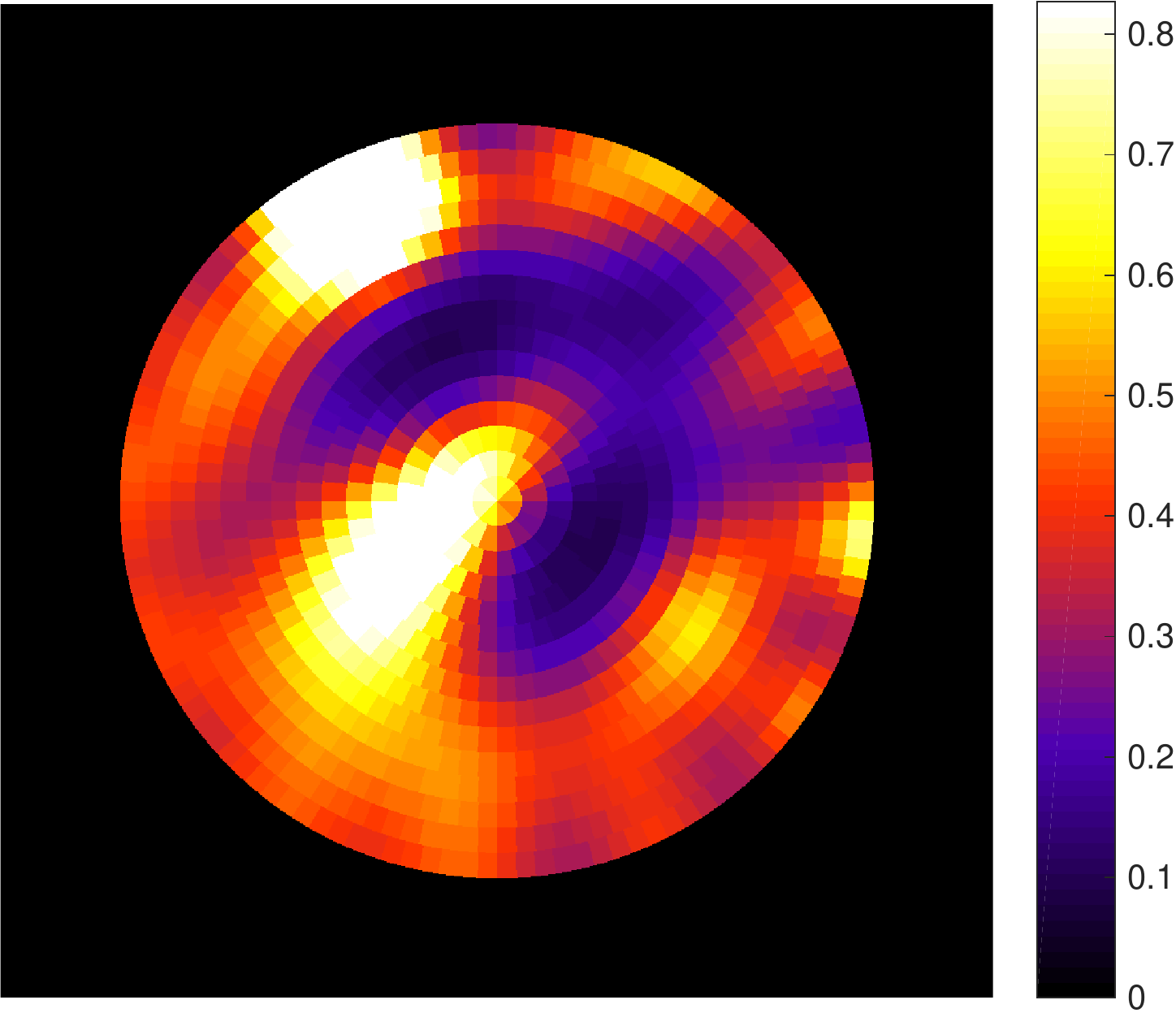}}
\qquad
{\includegraphics[height=2in]{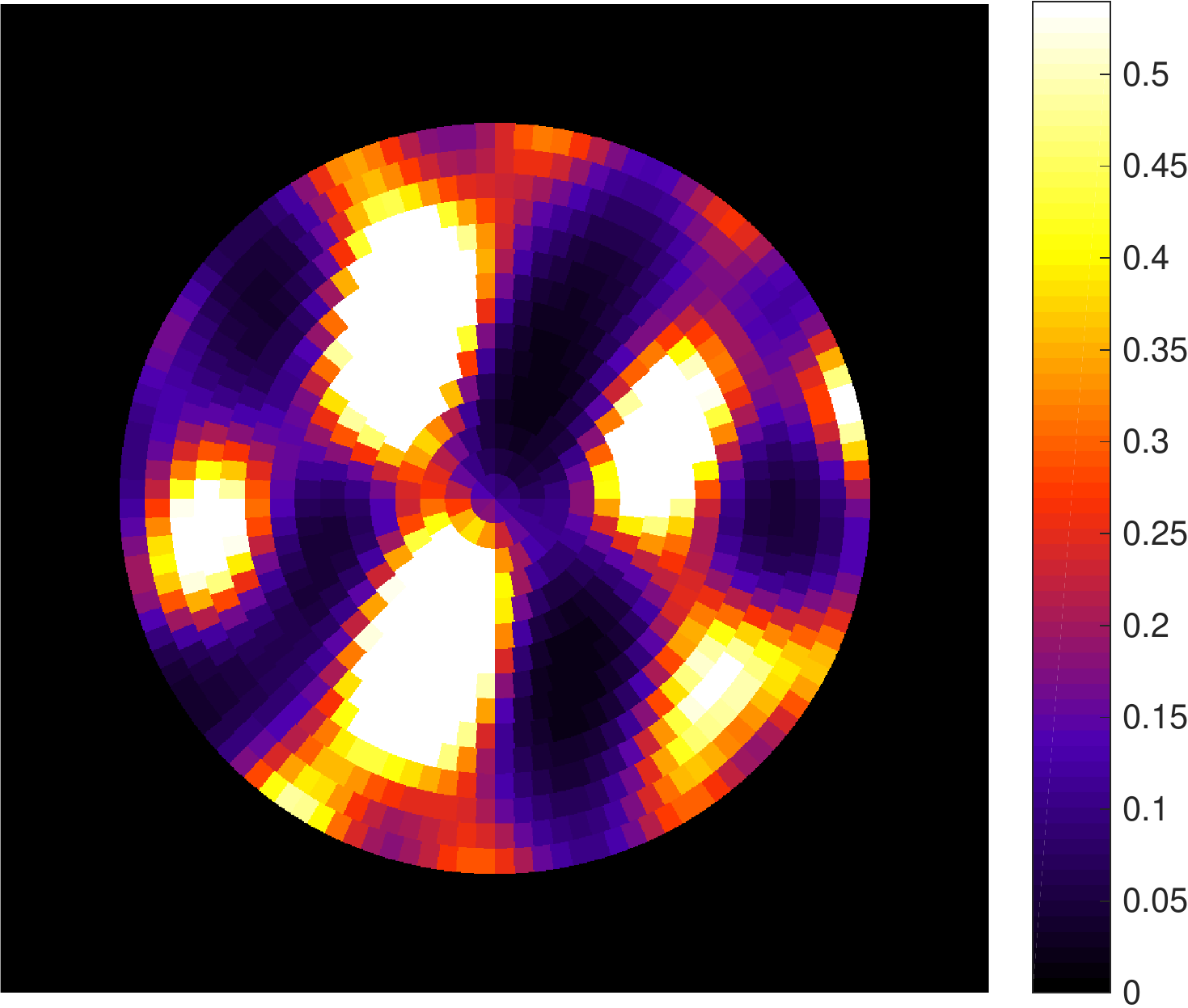}} 
} 
\caption{Simulated data. The reconstructions of conductivity obtained by fixing the domain to be a disk of radius $\rho_0 = 17.5$ with equally spaced electrodes attached to its boundary. The left-hand image corresponds to the target in the top row of Figure~\ref{fig:simulated1} and the right-hand image to that in the bottom row of Figure~\ref{fig:simulated1}. The color axes are those used in Figure~\ref{fig:simulated1}, but the highest reconstructed values are approximately $1.49$ (left) and $4.17$ (right).}
\label{fig:naive}
\end{figure}

\subsection{Experimental data}

We next apply our algorithm to four sets of experimental data from two water tanks: a thorax-shaped with circumference $106\,{\rm cm}$ and a deformable one with circumference $86 \, {\rm cm}$ (cf.~Figures~\ref{fig:tank1} and \ref{fig:tank2}). Both tanks have $M=16$ rectangular metallic electrodes of width $2\,{\rm cm}$ attached to their interior lateral surface.  In each experiment, the considered water tank contains vertically homogeneous embedded cylinders of steel and/or plastic extending from the bottom all the way through the water surface. The water level is controlled so that the tanks are always filled with tap water up to the top of the electrodes, which are of height $5 \, {\rm cm}$ for the thorax-shaped tank and of height $7 \, {\rm cm}$ for the deformable one. The measurements were performed with low-frequency ($1\,{\rm kHz}$) alternating current using the {\em Kuopio impedance tomography} (KIT4) device~\cite{Kourunen09}. The phase information of the measurements is ignored, meaning that the amplitudes of electrode currents and potentials are interpreted as real numbers.  

\begin{figure}
\center{
{\includegraphics[height=2in]{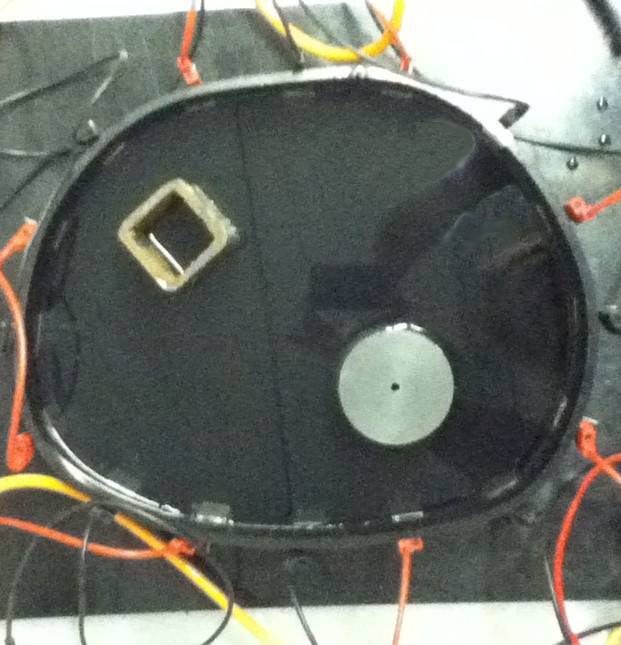}}
\qquad
{\includegraphics[height=2in]{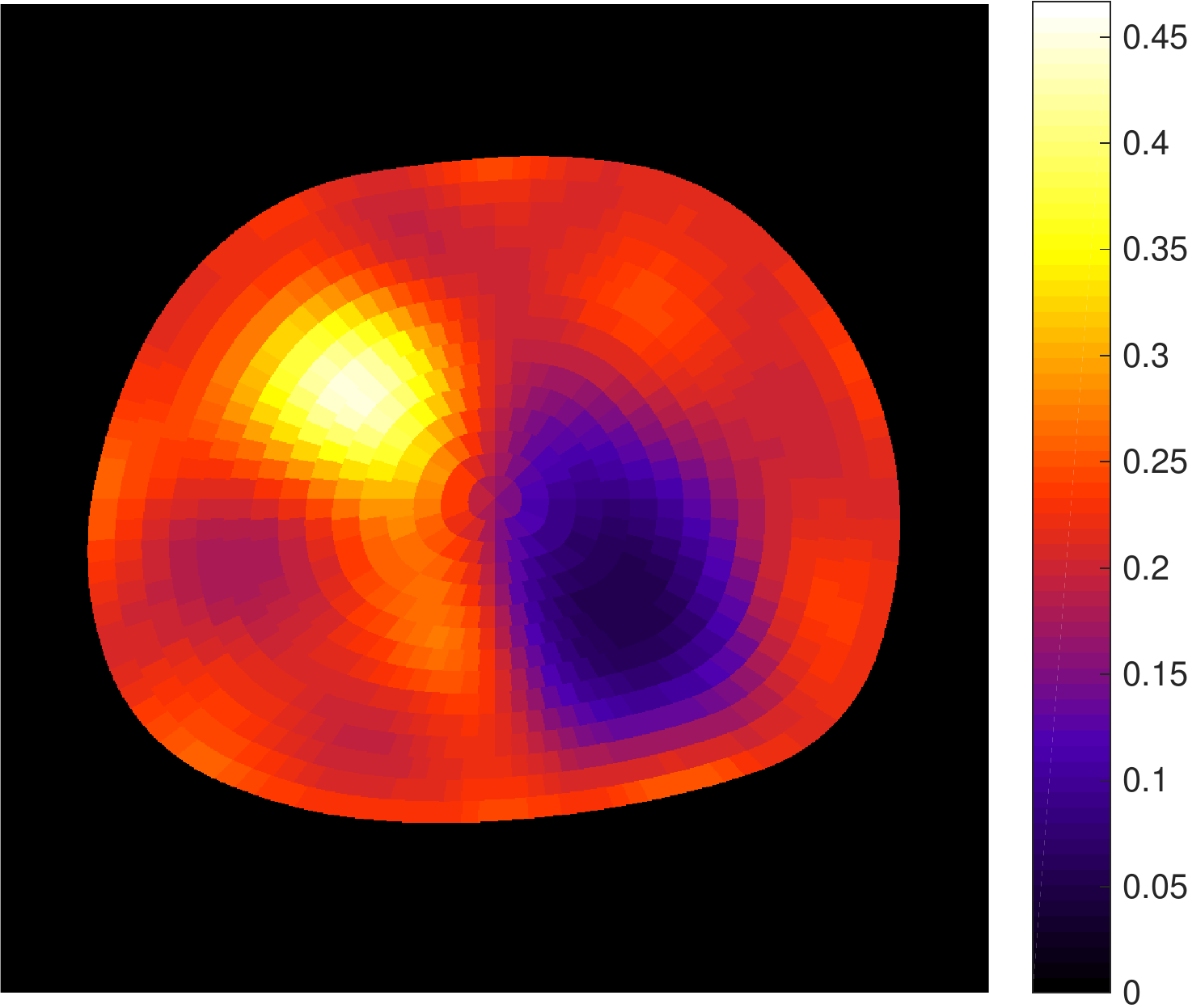}} 
} \\
\hspace{3mm}
\center{
{\includegraphics[height=2in]{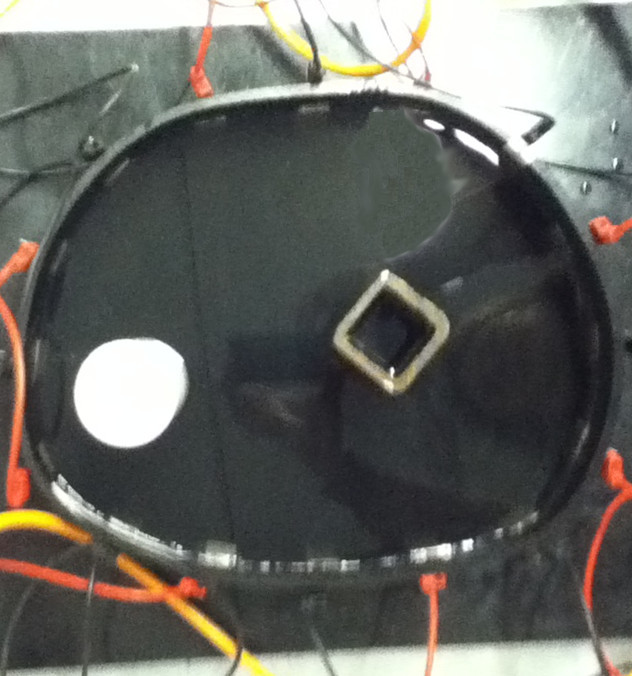}}
\qquad
{\includegraphics[height=2in]{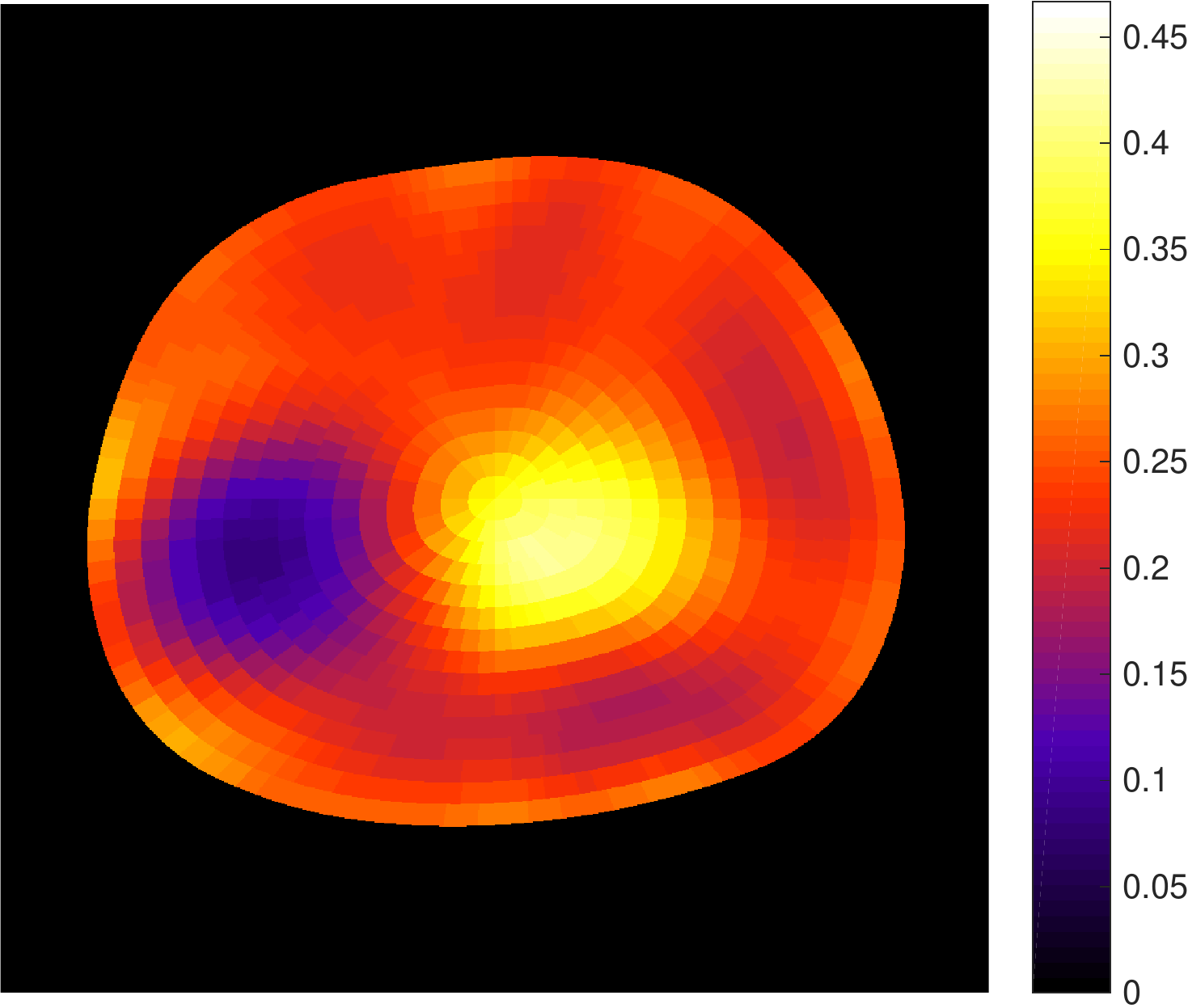}}
}
\caption{Experimental data from a thorax-shaped tank. Left: the target configurations. Right: the reconstructions. The unit of conductivity is ${\rm mS}/{\rm cm}$. The images are not in scale; the circumferences of the tank and the reconstructions are given in the text.}
\label{fig:tank1}
\end{figure}

As the measurement configurations are vertically homogeneous and no current flows through the top or the bottom of the water layer, one can employ the two-dimensional version of the CEM as the forward model. For a discussion on the conversion of units between two and three spatial dimensions, see,~e.g.,~\cite{Hyvonen15}. To put it short, if the voltage measurements on the electrodes are multiplied or, alternatively, the net currents are divided by the height of the tank, a two-dimensional inverse solver automatically produces reconstructions in the proper three-dimensional units. We take here the latter approach based on the (three-dimensional) current patterns~\eqref{current_input}; prior to the scaling by the tank height, the unit of current is mA. The regularization parameter is chosen as
$$
\lambda = 2\cdot 10^{-3} \max_{j,k=1,\dots, Q} \big| \mathcal{V}_j - \mathcal{V}_k \big|
$$
for both tanks.
Loosely speaking, the Bayesian interpretation of this is that we expect roughly $0.1\,\%$ of measurement noise as in the case of simulated data.

\begin{figure}
\center{
{\includegraphics[height=2in]{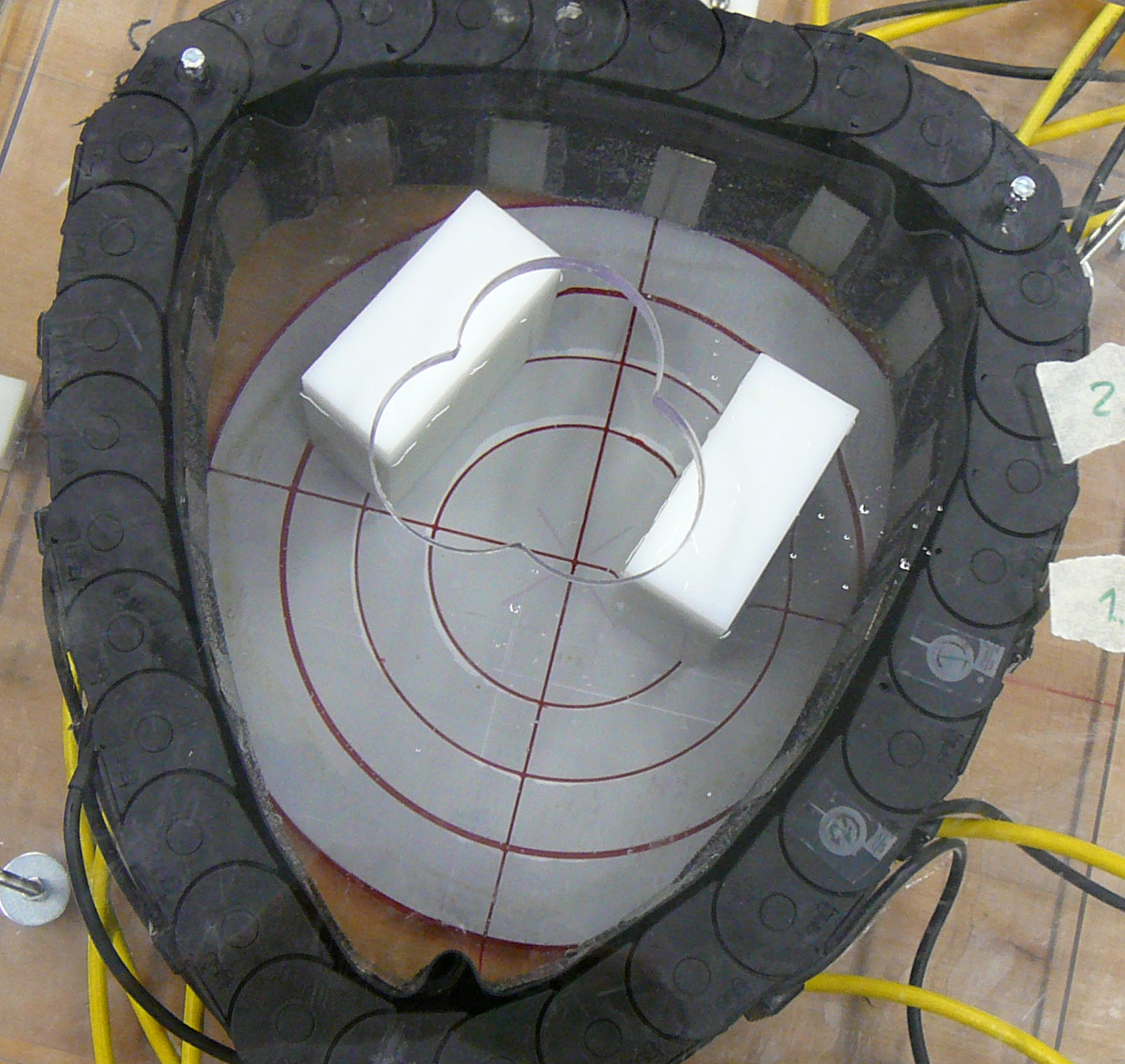}}
\qquad
{\includegraphics[height=2in]{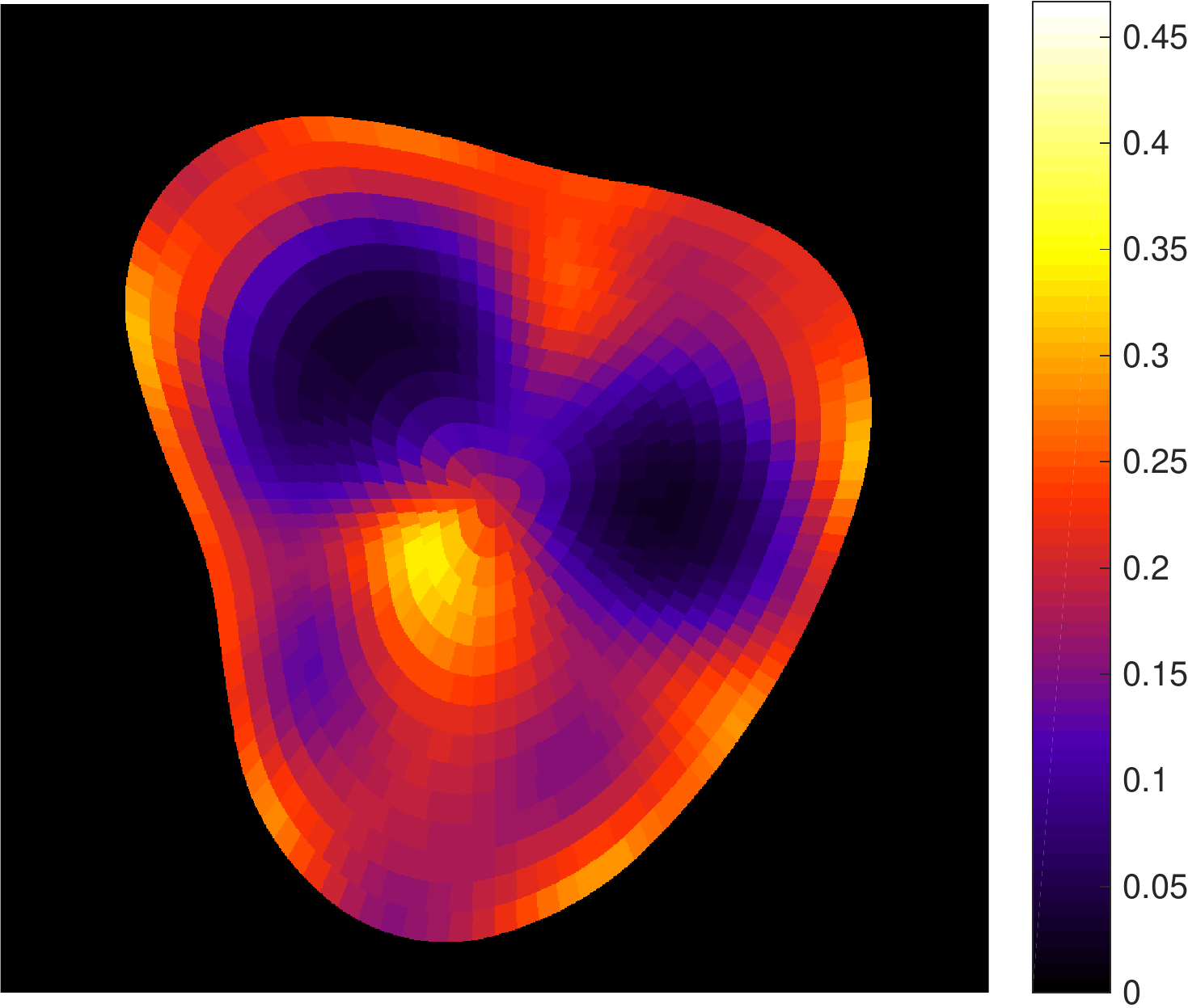}}\\
}
\hspace{3mm}
\center{
{\includegraphics[height=2in]{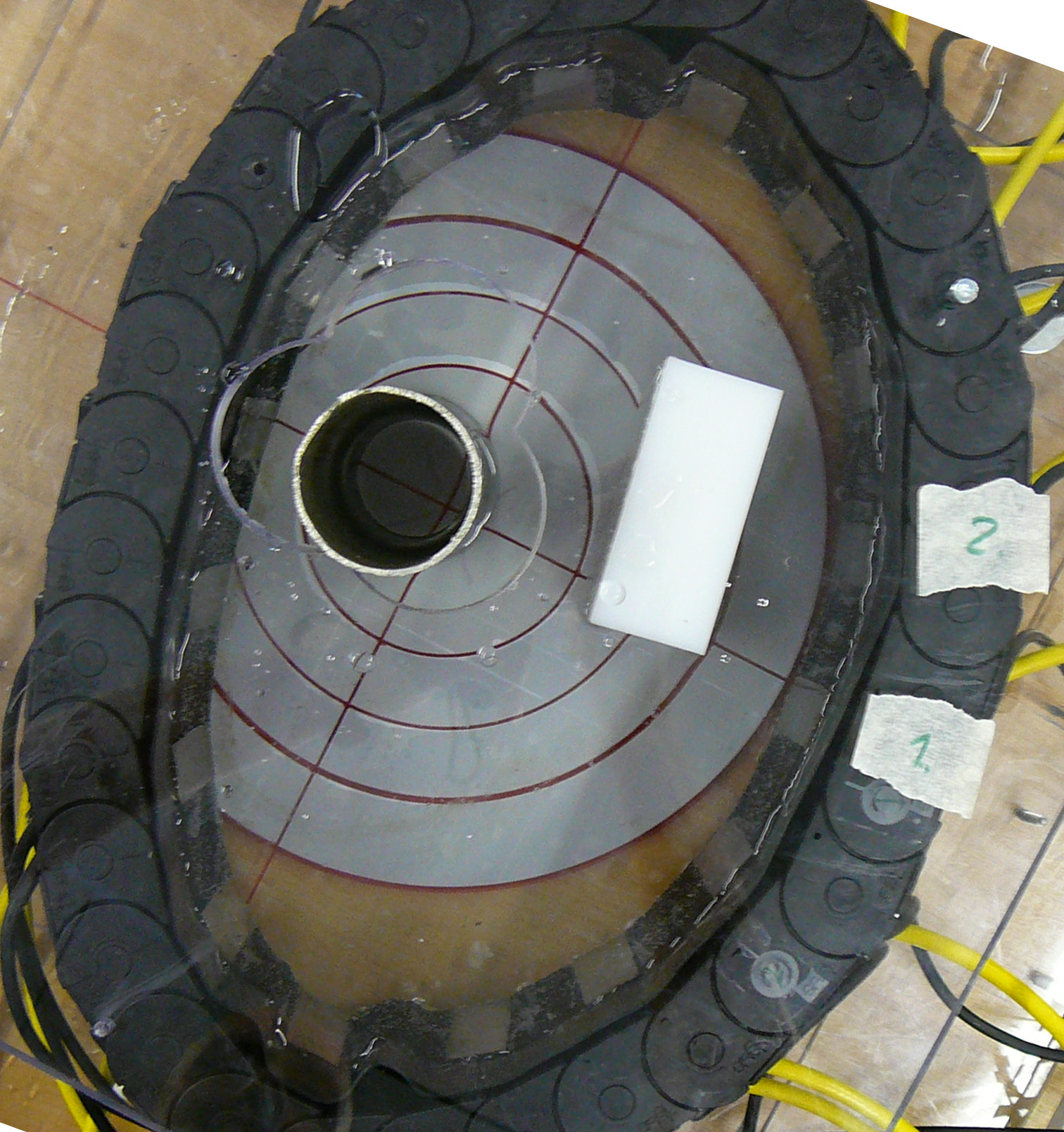}}
\qquad
{\includegraphics[height=2in]{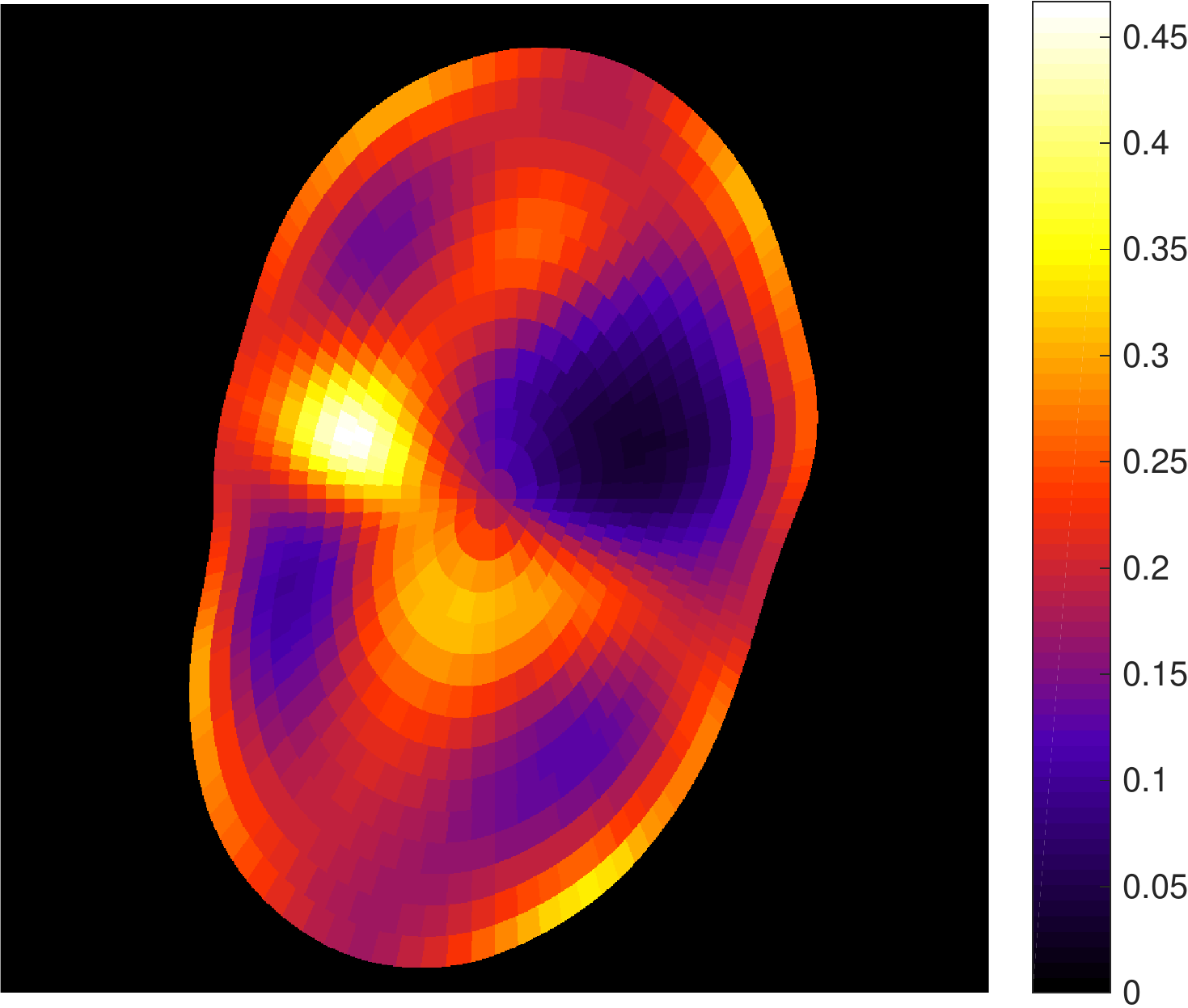}}
}
\caption{Experimental data from a deformable tank. Left: the target configurations. Right: the reconstructions. The unit of conductivity is ${\rm mS}/{\rm cm}$. The images are not in scale; the circumferences of the tank and the reconstructions are given in the text.}
\label{fig:tank2}
\end{figure}

The conductivity reconstructions corresponding to the thorax-shaped tank are presented in Figure~\ref{fig:tank1}.
The two target configurations are shown in the left-hand column and the corresponding reconstructions on the right. In both cases, there are two cylinders placed inside the tank: a metallic one with square cross-section and a plastic one with round cross-section. For both targets, the reconstruction of the tank boundary is accurate and the approximate positions of the inclusions can also be deduced from the images in the right-hand column of Figure~\ref{fig:tank1}. However, especially the insulating inclusions appear blurred in the reconstructions and there are also some oscillations in the estimated background conductivity level.
Both reconstructions have a circumference of about $112\,{\rm cm}$.

The reconstructions corresponding to the deformable tank are presented in Figure~\ref{fig:tank2},
which is organized in the same way as Figure~\ref{fig:tank1}: The target configurations are shown in the left-hand column and the associated reconstructions in the right-hand column. The water tank has been bent into two different shapes. The corresponding conductivity phantoms consist of two pieces of plastic with rectangular cross-sections and of one round steel cylinder and a rectangular body of plastic, respectively.
The reconstructions are not quite as informative as for the thorax-shaped tank, which is inline with our experience of expecting data from the deformable tank to be of lower quality.
The shapes of the exterior boundary are not reproduced as accurately, the inclusions appear more blurred and the variations in the background conductivity level are notable. It seems that the algorithm tries to explain some of the data variations originating from the inhomogeneities by deforming the object boundary.
The reconstruction circumferences $116\,{\rm cm}$ (top) and $114\,{\rm cm}$ (bottom) are also quite far off the mark;
our hypothesis is that the algorithm compensates for the overestimation of the tank size by downtuning the contact resistances (cf.~the top right image of Figure~\ref{fig:simulated2}).
In any case, the reconstructions in Figure~\ref{fig:tank2} still carry useful information about the corresponding targets. In particular, they are far better than ones obtained by altogether ignoring the incompleteness of the information on the measurement geometry and computing the conductivity reconstruction in, e.g., a disk (cf.~\cite{Darde13b}):
In Figure~\ref{fig:tank2naive} we present the reconstructions similar to those in Figure~\ref{fig:naive}.
Now the geometry is fixed to a disk having the correct circumference of $86 \, {\rm cm}$ and equiangled electrodes.

\begin{figure}
\center{
{\includegraphics[height=2in]{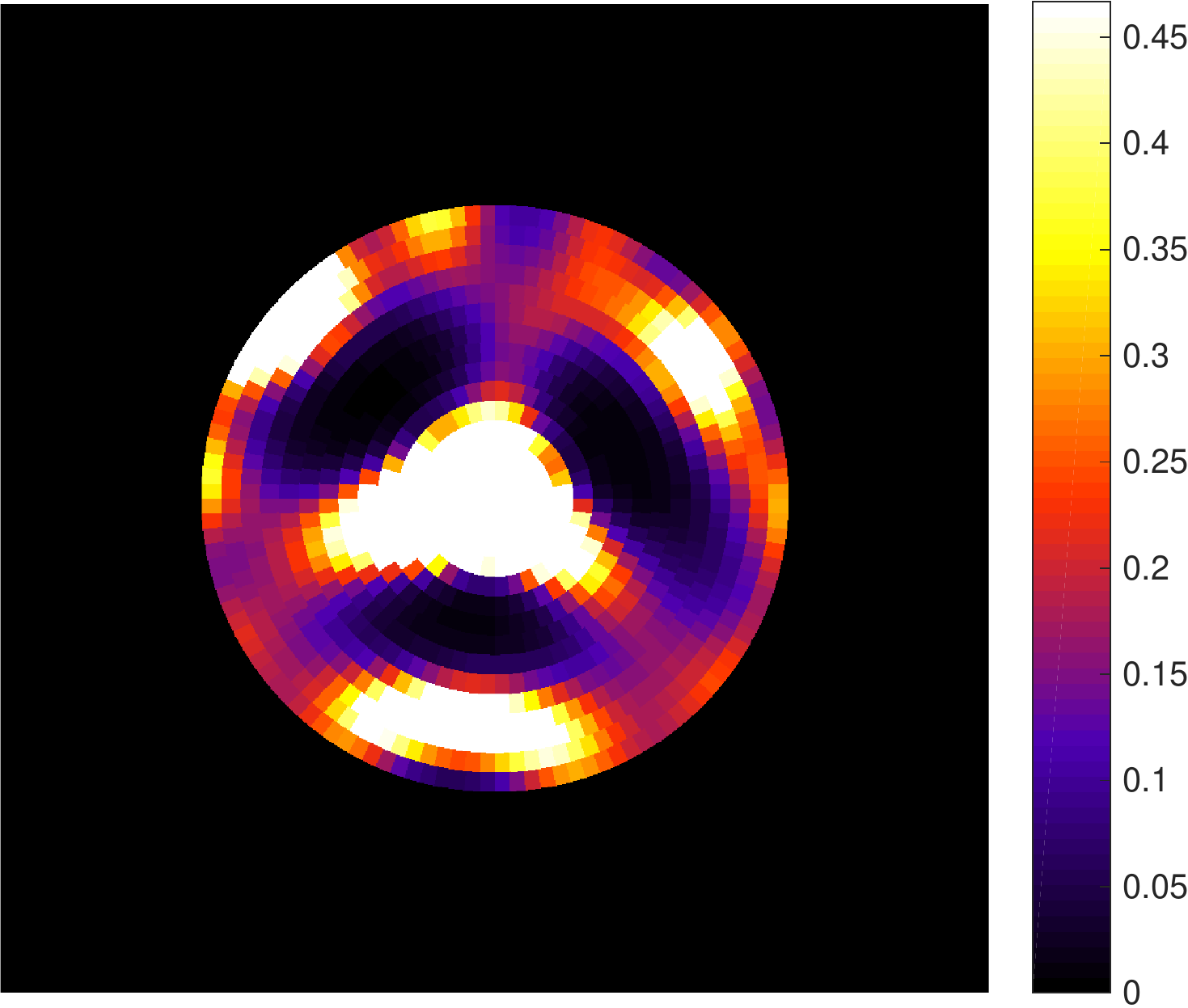}}
\qquad
{\includegraphics[height=2in]{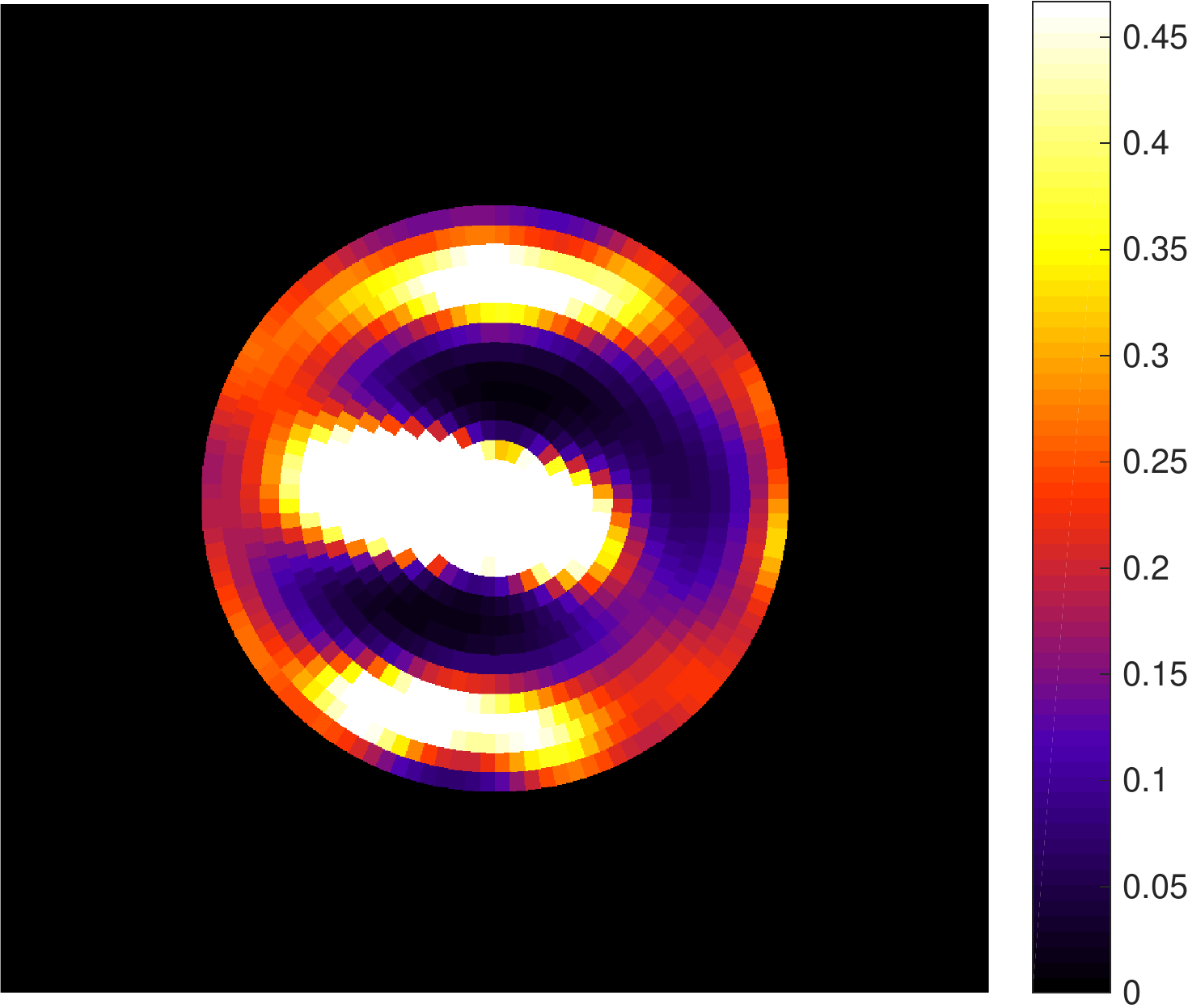}}
}
\caption{Reconstructions based on the data from a deformable tank and fixed geometry. The left image corresponds to the top row of Figure~\ref{fig:tank2} and the right image to the bottom row. The highest reconstructed values are approximately $34.2$ (left) and $14.6$ (right).}
\label{fig:tank2naive}
\end{figure}

\section{Conclusion}
\label{sec:conclusion}
By employing cFEM, we have introduced a numerical algorithm that is capable of simultaneously producing reasonable reconstructions of the conductivity 
and the exterior boundary shape of the examined body in EIT from both simulated and experimental data.
The algorithm consists of two stages: In the pre-measurement processing, a polynomial surrogate model is formed for the CEM. This is the computationally expensive part of the proposed method, but it can fortunately be carried out off-line prior to any actual measurements, assuming there is enough general-level information available on the measurement set-up (the approximate size and conductivity level of the body, the shape and number of the electrodes etc.). The post-measurement processing consists of minimizing a sum of squares of multivariate polynomials, which does not demand a lot of computation time --- unless the polynomial order 
in the surrogate model is high. Our numerical experiments were based on second-degree polynomials and approximately a thousand parameters, which resulted in post-processing times of only a few seconds.

\section*{Acknowledgments}
We would like to thank Professor Jari Kaipio's research group at the University of Eastern Finland (Kuopio) for granting us access to their EIT devices.

\bibliographystyle{acm}
\bibliography{sceit-refs}

\begin{thebibliography}{10}

\bibitem{Adams03}
{\sc Adams, R.~A., and Fournier, J. J.~F.}
\newblock {\em Sobolev spaces}, second~ed., vol.~140 of {\em Pure and Applied
  Mathematics}.
\newblock Elsevier/Academic Press, Amsterdam, 2003.

\bibitem{Alberti16}
{\sc Alberti, G.~S., Ammari, H., Jin, B., Seo, J.-K., and Zhang, W.}
\newblock The linearized inverse problem in multifrequency electrical impedance
  tomography, 2016.
\newblock arXiv:1602.04312.

\bibitem{Babuska10}
{\sc Babu\v{s}ka, I., Nobile, F., and Tempone, R.}
\newblock A stochastic collocation method for elliptic partial differential
  equations with random input data.
\newblock {\em SIAM Rev. 52\/} (2010), 317--355.

\bibitem{Barber84}
{\sc Barber, D.~C., and Brown, B.~H.}
\newblock Applied potential tomography.
\newblock {\em J. Phys. E: Sci. Instrum. 17\/} (1984), 723--733.

\bibitem{Barber88}
{\sc Barber, D.~C., and Brown, B.~H.}
\newblock Errors in reconstruction of resistivity images using a linear
  reconstruction technique.
\newblock {\em Clin. Phys. Physiol. Meas. 9\/} (1988), 101--104.

\bibitem{Borcea02}
{\sc Borcea, L.}
\newblock Electrical impedance tomography.
\newblock {\em Inverse problems 18\/} (2002), R99--R136.

\bibitem{Boulkhemair07}
{\sc Boulkhemair, A., and Chakib, A.}
\newblock On the uniform {P}oincar\'{e} inequality.
\newblock {\em Comm. Partial Differential Equations 32\/} (2007), 1439--1447.

\bibitem{Breckon88}
{\sc Breckon, W., and Pidcock, M.}
\newblock Data errors and reconstruction algorithms in electrical impedance
  tomography.
\newblock {\em Clin. Phys. Physiol. Meas. 9\/} (1988), 105--109.

\bibitem{Cheney99}
{\sc Cheney, M., Isaacson, D., and Newell, J.}
\newblock Electrical impedance tomography.
\newblock {\em SIAM Rev. 41\/} (1999), 85--101.

\bibitem{Cheng89}
{\sc Cheng, K.-S., Isaacson, D., Newell, J.~S., and Gisser, D.~G.}
\newblock Electrode models for electric current computed tomography.
\newblock {\em IEEE Trans. Biomed. Eng. 36\/} (1989), 918--924.

\bibitem{Costabel96}
{\sc Costabel, M., and Dauge, M.}
\newblock A singularly mixed perturbed boundary value problem.
\newblock {\em Comm. Partial Differential Equations 21\/} (1996), 1919--1949.

\bibitem{Darde12}
{\sc Dard\'e, J., Hakula, H., Hyv\"onen, N., and Staboulis, S.}
\newblock Fine-tuning electrode information in electrical impedance tomography.
\newblock {\em Inverse Probl. Imag. 6\/} (2012), 399--421.

\bibitem{Darde13a}
{\sc Dard\'e, J., Hyv\"onen, N., Sepp\"anen, A., and Staboulis, S.}
\newblock Simultaneous reconstruction of outer boundary shape and admittance
  distribution in electrical impedance tomography.
\newblock {\em SIAM J. Imaging Sci. 6\/} (2013), 176--198.

\bibitem{Darde13b}
{\sc Dard{\'e}, J., Hyv{\"o}nen, N., Sepp{\"a}nen, A., and Staboulis, S.}
\newblock Simultaneous recovery of admittivity and body shape in electrical
  impedance tomography: An experimental evaluation.
\newblock {\em Inverse Problems 29\/} (2013), 085004.

\bibitem{Darde16}
{\sc Dard\'e, J., and Staboulis, S.}
\newblock Electrode modelling: The effect of contact impedance.
\newblock {\em ESAIM: Math. Model. Num. 50\/} (2016), 415--431.

\bibitem{Gautschi04}
{\sc Gautschi, W.}
\newblock {\em Orthogonal Polynomials: Computation and Approximation}.
\newblock Numerical Mathematics and Scientific Computation. Oxford University
  Press, New York, 2004.
\newblock Oxford Science Publications.

\bibitem{Hakula14}
{\sc Hakula, H., Hyv\"onen, N., and Leinonen, M.}
\newblock Reconstruction algorithm based on stochastic {G}alerkin finite
  element method for electrical impedance tomography.
\newblock {\em Inverse Problems 30\/} (2014), 065006.

\bibitem{Hanke11b}
{\sc Hanke, M., Harrach, B., and Hyv{\"o}nen, N.}
\newblock Justification of point electrode models in electrical impedance
  tomography.
\newblock {\em Math. Models Methods Appl. Sci. 21\/} (2011), 1395--1413.

\bibitem{Hiptmair15}
{\sc Hiptmair, R., Scarabosio, L., Schillings, C., and Schwab, {\relax Ch}.}
\newblock Large deformation shape uncertainty quantification in acoustic
  scattering.
\newblock Tech. Rep. No.~2015-31, Seminar for Applied Mathematics, ETH, 2015.

\bibitem{Hollig03}
{\sc H{\"o}llig, K., and H{\"o}rner, J.}
\newblock {\em Finite Element Methods with B-Splines}, vol.~26 of {\em
  Frontiers in Applied Mathematics}.
\newblock Society for Industrial and Applied Mathematics (SIAM), Philadelphia,
  PA, 2003.

\bibitem{Hyvonen04}
{\sc Hyv\"onen, N.}
\newblock Complete electrode model of electrical impedance tomography:
  Approximation properties and characterization of inclusions.
\newblock {\em SIAM J.~App.~Math. 64\/} (2004), 902--931.

\bibitem{Hyvonen09}
{\sc Hyv{\"o}nen, N.}
\newblock Approximating idealized boundary data of electric impedance
  tomography by electrode measurements.
\newblock {\em Math. Models Methods Appl. Sci. 19\/} (2009), 1185--1202.

\bibitem{Hyvonen15}
{\sc Hyv\"onen, N., and Leinonen, M.}
\newblock Stochastic {G}alerkin finite element method with local conductivity
  basis for electrical impedance tomography.
\newblock {\em SIAM/ASA J. Uncertainty Quantification 3\/} (2015), 998--1019.

\bibitem{Kaipio05}
{\sc Kaipio, J., and Somersalo, E.}
\newblock {\em Statistical and Computational Inverse Problems}.
\newblock Springer, 2005.

\bibitem{Kaipio04}
{\sc Kaipio, J.~P., Sepp\"anen, A., Somersalo, E., and Haario, H.}
\newblock Posterior covariance related optimal current patterns in electrical
  impedance tomography.
\newblock {\em Inverse Problems 20\/} (2004), 919--936.

\bibitem{Kolehmainen05}
{\sc Kolehmainen, V., Lassas, M., and Ola, P.}
\newblock Inverse conductivity problem with an imperfectly known boundary.
\newblock {\em SIAM J. App. Math. 66\/} (2005), 365--383.

\bibitem{Kolehmainen07}
{\sc Kolehmainen, V., Lassas, M., and Ola, P.}
\newblock The inverse conductivity problem with an imperfectly known boundary
  in three dimensions.
\newblock {\em SIAM J. Appl. Math. 67\/} (2007), 1440--1452.

\bibitem{Kolehmainen97}
{\sc Kolehmainen, V., Vauhkonen, M., Karjalainen, P.~A., and Kaipio, J.~P.}
\newblock Assessment of errors in static electrical impedance tomography with
  adjacent and trigonometric current patterns.
\newblock {\em Physiol. Meas. 18\/} (1997), 289--303.

\bibitem{Kourunen09}
{\sc Kourunen, J., Savolainen, T., Lehikoinen, A., Vauhkonen, M., and
  Heikkinen, L.~M.}
\newblock Suitability of a {PXI} platform for an electrical impedance
  tomography system.
\newblock {\em Meas. Sci. Technol. 20\/} (2009), 015503.

\bibitem{Lions72}
{\sc Lions, J.~L., and Magenes, E.}
\newblock {\em Non-homogeneous boundary value problems and applications},
  vol.~1.
\newblock Springer-Verlag, 1973.
\newblock Translated from French by P. Kenneth.

\bibitem{Mustonen15}
{\sc Mustonen, L.}
\newblock Numerical study of a parametric parabolic equation and a related
  inverse boundary value problem.
\newblock Submitted, arXiv:1506.01559.

\bibitem{Nissinen11}
{\sc Nissinen, A., Kolehmainen, V., and Kaipio, J.~P.}
\newblock Compensation of modelling errors due to unknown domain boundary in
  electrical impedance tomography.
\newblock {\em IEEE Trans. Med. Imag. 30\/} (2011), 231--242.

\bibitem{Nissinen11b}
{\sc Nissinen, A., Kolehmainen, V., and Kaipio, J.~P.}
\newblock Reconstruction of domain boundary and conductivity in electrical
  impedance tomography using the approximation error approach.
\newblock {\em Int. J. Uncertainty Quantif. 1\/} (2011), 203--222.

\bibitem{Nocedal99}
{\sc Nocedal, J., and Wright, S.~J.}
\newblock {\em Numerical Optimization}.
\newblock Springer, 1999.

\bibitem{NovRit96}
{\sc Novak, E., and Ritter, K.}
\newblock High dimensional integration of smooth functions over cubes.
\newblock {\em Numer. Math. 75}, 1 (1996), 79--97.

\bibitem{NovRit99}
{\sc Novak, E., and Ritter, K.}
\newblock Simple cubature formulas with high polynomial exactness.
\newblock {\em Constr. Approx. 15}, 4 (1999), 499--522.

\bibitem{Rossi08}
{\sc Rossi, J.~D.}
\newblock First variations of the best {S}obolev trace constant with respect to
  the domain.
\newblock {\em Canad. Math. Bull. 51}, 1 (2008), 140--145.

\bibitem{Schwab11a}
{\sc Schwab, {\relax Ch}., and Gittelson, C.~J.}
\newblock Sparse tensor discretizations of high-dimensional parametric and
  stochastic {PDE}s.
\newblock {\em Acta Numer. 20\/} (2011), 291--467.

\bibitem{Schwab12}
{\sc Schwab, {\relax Ch}., and Stuart, A.~M.}
\newblock Sparse deterministic approximation of {B}ayesian inverse problems.
\newblock {\em Inverse Problems 28\/} (2012), 045003.

\bibitem{Smolyak63}
{\sc Smolyak, S.}
\newblock Quadrature and interpolation formulas for tensor products of certain
  classes of functions.
\newblock {\em Soviet Mathematics 4\/} (1963), 240--243.

\bibitem{Somersalo92}
{\sc Somersalo, E., Cheney, M., and Isaacson, D.}
\newblock Existence and uniqueness for electrode models for electric current
  computed tomography.
\newblock {\em SIAM J. Appl. Math. 52\/} (1992), 1023--1040.

\bibitem{Uhlmann09}
{\sc Uhlmann, G.}
\newblock Electrical impedance tomography and {C}alder{\'o}n's problem.
\newblock {\em Inverse Problems 25\/} (2009), 123011.

\bibitem{Vilhunen02}
{\sc Vilhunen, T., Kaipio, J.~P., Vauhkonen, P.~J., Savolainen, T., and
  Vauhkonen, M.}
\newblock Simultaneous reconstruction of electrode contact impedances and
  internal electrical properties: {I}. {T}heory.
\newblock {\em Meas. Sci. Technol. 13\/} (2002), 1848–1854.

\bibitem{WasilWoz95}
{\sc Wasilkowski, G., and Wo\`{z}niakowski, H.}
\newblock Explicit cost bounds of algorithms for multivariate tensor product
  problems.
\newblock {\em J. Complexity 11\/} (1995), 1--56.

\bibitem{Xiu06}
{\sc Xiu, D.}
\newblock Efficient collocational approach for parametric uncertainty analysis.
\newblock {\em Commun. Comput. Phys 2}, 2 (2006), 293--309.

\end{thebibliography}
\end{document}